\documentclass[11pt]{amsart}
\usepackage{amsmath,amsthm, amscd, amssymb, amsfonts, mathtools}
\usepackage[all]{xy}
\usepackage{enumitem}
\usepackage{hyperref}
\usepackage{t1enc}
\usepackage[latin1]{inputenc}
\usepackage{fontsmpl}
\usepackage{srcltx}
\usepackage{placeins}
\usepackage{float}
\usepackage[dvipsnames]{xcolor}

\allowdisplaybreaks

\numberwithin{equation}{section}\theoremstyle{plain}

\newtheorem{thm}{Theorem}[section]
\newtheorem{prop}[thm]{Proposition}
\newtheorem{lem}[thm]{Lemma}
\newtheorem{cor}[thm]{Corollary}
\theoremstyle{definition}

\newtheorem{rem}[thm]{Remark}

\newtheorem{exa}[thm]{Example}
\newtheorem{exas}[thm]{Examples}

\newcommand\I{\mathbb I}

\def \N {\mathbb{N}}
\def \Z {\mathbb{Z}}
\def \k {\Bbbk}
\def \o {\otimes}

\begin{document}


 \title[Partial (co)actions of Taft and Nichols Hopf algebras on algebras]{Partial (co)actions of Taft and Nichols Hopf algebras on algebras}
\author[Fonseca, Martini and Silva]{Graziela Fonseca, Grasiela Martini and Leonardo Silva}

\address[Fonseca]{Instituto Federal Sul-rio-grandense, Brazil}
\email{grazielalangone@gmail.com}

\address[Martini]{Universidade Federal do Rio Grande do Sul, Brazil}
\email{grasiela.martini@ufrgs.br}

\address[Silva]{Universidade Federal do Rio Grande do Sul, Brazil.}
\email{dsleonardo@ufrgs.br}

\thanks{
	\hspace*{0.3cm} 
}

\begin{abstract}
	In this paper, we characterize suitable partial (co)actions of Taft and Nichols Hopf algebras on algebras, and moreover we get that such partial (co)actions are symmetric.
	For certain algebras, these partial (co)actions obtained are, indeed, all of them.
	This work generalizes the results obtained by the authors in \cite{taft_corpo_revista}.
\end{abstract}

\maketitle

\tableofcontents

\section{Introduction}

Global and partial actions of Hopf algebras on algebras are widely studied subjects, with various interests, applications and ramifications. However, calculating and classifying such actions is a difficult task. There exist several works that calculate (global) actions of certain Hopf algebras on different types of algebras, such as fields, domains, Weyl algebras, path algebras of quivers, among others \cite{CuadraEtingofWaltonI,CuadraEtingofWaltonII, etingofewaltonSemisimple, etingofewaltonI, etingofewaltonFinite, etingofewaltonFinite2, etingofewaltonII}. 

In general, the works investigate actions of two classes of Hopf algebras: the semisimple and the pointed (nonsemisimple).
In particular, the Taft algebra is an important example to be considered in the pointed case.
Because of that, some works deal particularly with Taft actions on algebras \cite{Acoes_taft_Centrone, Acoes_taft_Chelsea,  montgomery_Schneider_taft_actions}.

In order to understand the actions of pointed Hopf algebras, another interesting example to be considered is the Nichols Hopf algebra: it is the archetype of pointed Hopf algebras.
Indeed, finite dimensional pointed Hopf algebras (with abelian group of group-like elements $G$) are expected to be a lifting of the bosonization of a Nichols algebra by the group algebra $\k G$ \cite{Nicolas_ICM}. In particular, the Nichols Hopf algebra is just the bosonization of an exterior algebra by the group algebra of the cyclic group of order 2.

For partial Hopf actions, the situation is even more unexplored, with just a few concrete examples known.
In \cite{taft_corpo_revista}, the authors exhibit explicitly all partial actions of Taft and Nichols Hopf algebras on their base fields, obtaining that such partial actions depend only on the partial action of the group of group-like elements, a derivation and a compatibility between both, same as it happens in the global case \cite{Acoes_taft_Centrone}.

In this work, the authors continue the investigation of the partial actions of Taft and Nichols Hopf algebras on algebras, extending the results of \cite{taft_corpo_revista}.
In particular, Theorem \ref{taft_principal} shows that, with suitable conditions, the partial actions still depend only on the partial group action, a derivation  and a compatibility between them.

This paper is organized as follows.
In Section \ref{sec:preliminaries}, we recall the definitions of partial (co)actions of a bialgebra on an algebra and also some definitions and results about \emph{q-combinatorics}, that will be used in the computations of the partial Taft actions.
In Section \ref{sec:Taft}, we investigate and characterize suitable partial actions of the Taft algebra on an algebra. Furthermore, we present the partial Taft coactions obtained dually from such partial actions. At last, in Section \ref{sec:Nichols}, we compute some partial (co)actions of a Nichols Hopf algebra on algebras.

\section{Preliminaries}\label{sec:preliminaries}

Throughout the work, \emph{algebra} means \emph{unital associative algebra}.
Besides, $\k^{\times} = \k \backslash \{0\}$ and unadorned $\otimes$ means $\otimes_{\Bbbk}$.	
For a coalgebra $C$, $\varepsilon_C$ and $\Delta_C$ stands for the counit and the comultiplication maps of $C$, respectively, and will be written as  $\varepsilon$ and $\Delta$ if there is no ambiguity.
Moreover, we use Sweedler-Heynemann notation with summation sign suppressed for the comultiplication map, namely $\Delta(c)=c_1 \o c_2 \in C \o C$, for all $c \in C$.
We write $G(C)= \{g \in C \backslash \{0\} \ : \ \Delta(g) = g \o g \}$ for the set of the \emph{group-like elements} of the coalgebra $C$.
Given $g,h \in G(C)$, an element $x \in C$ is called a \emph{$(g,h)$-primitive element} if $\Delta(x)= x \o g + h \o x$.
If no emphasis on the elements $g, h \in G(C)$ is needed, a $(g,h)$-primitive element will be called simply of a \emph{skew-primitive element}.

For a group $G$, we write $\k G$ for the group algebra with the canonical basis $\{g \ : \ g \in G \}$.
Moreover, given a basis $\{v_1, v_2, \cdots , v_n\}$ of a finite-dimensional vector space, we write $ \{v_1^*, v_2^*, \cdots , v_n^*\}$ for its dual basis.

We use $\Z$, $\N$ and $\N_0$ for the sets of integers, the positive integers and $\N \cup \{0\}$, respectively.
	Let $ j, k \in \Z$ and $n \in \N$.
	The symbols $\delta_{j,k}$ and $C_n$ stand for Kronecker's delta and the cyclic group of order $n$, respectively.
	If $j \leq k$, then $\I_{j, k} =\{j, j+1, \cdots, k\}$ and $\I_n = \I_{1, n}$.

\subsection{Partial (co)actions}
Partial actions and partial coactions of bialgebras on algebras were introduced in \cite{caenepeel2008partial}.
Since this earlier definition does not deal with symmetric conditions, we will present as it appears in later works.
These conditions are important for some developments in the theory of partial (co)actions, such as partial (co)representations \cite{corepresentations, Partial_representations, dilations}.

A \emph{partial action of a bialgebra $H$ on an algebra $A$} is a linear map $\cdot: H \otimes A \longrightarrow A$, denoted by $\cdot (h \o a) = h \cdot a$, such that
\begin{itemize}
	\item[(PA.1)]\label{PA1} $1_H \cdot a=a$, 
	\item[(PA.2)]\label{PA2} $h\cdot ab=(h_1\cdot a)(h_2\cdot b)$, and
	\item[(PA.3)]\label{PA3} $h\cdot(k\cdot a)=(h_1\cdot 1_A)(h_2k\cdot a)$,
\end{itemize}
hold for all $h,k\in H$ and $a,b\in A$.
In this case, $A$ is called a \textit{partial $H$-module algebra}.
A partial action is \emph{symmetric} if in addition we have
\begin{itemize}
	\item[(PA.S)]\label{PAS} $h \cdot ( k \cdot a)=(h_1k \cdot a)(h_2 \cdot 1_A)$.
\end{itemize}

Note that conditions (PA.2) and (PA.3) are equivalent  to the following:
\begin{itemize}
	\item[(PA.2$^\prime$)]\label{PA2e3} $h \cdot (a( k \cdot b))=(h_1 \cdot a)(h_2k \cdot b)$.
\end{itemize}

Every $H$-module algebra is a symmetric partial $H$-module algebra. Moreover, a partial $H$-module algebra is an $H$-module algebra if and only if $h \cdot 1_A = \varepsilon(h) 1_A$, for all $h \in H$.

A \emph{(symmetric) partial coaction of a bialgebra $H$ on an algebra $A$} is a linear map $\rho: A \longrightarrow A\otimes H$ satisfying dual conditions to the previous conditions (PA.$\_$).
In fact, the concepts of partial actions and partial coactions are dually related in \cite{dual_constructions} as follows.

Let $\langle -,- \rangle: K \otimes H \longrightarrow \Bbbk$ be a dual pairing between bialgebras and $A$ an algebra. If $A$ is a (symmetric)
partial $K$-comodule algebra, then $A$ is a (symmetric) partial $H$-module algebra. 
Furthermore, if the dual pairing $\langle -,- \rangle$ is non-degenerate and $A$ is a rational (symmetric) partial $H$-module algebra, then $A$ is a (symmetric) partial $K$-comodule algebra.
Thus, in the particular case that $H$ is a finite-dimensional bialgebra with basis $\{h_i\}_{i \in I}$, we obtain the following characterization.

\begin{prop} \label{dual}
	If $A$ is a (symmetric) partial $H$-module algebra via $h \cdot a$, then $A$ is a (symmetric) partial $H^{\ast}$-comodule algebra via	$\rho: A \rightarrow A\otimes H^*$, given by $\rho(a)=\sum_{i \in I} h_i\cdot a\otimes h_i^{\ast}$. Reciprocally, if $A$ is a (symmetric) partial $H$-comodule algebra via $\rho: A \rightarrow A\otimes H$, where $\rho(a)=a^0\otimes a^1$, then $A$ is a (symmetric) partial $H^{\ast}$-module algebra via $h^{\ast} \cdot a=h^{\ast}(a^1)a^0$.
\end{prop}

\subsection{q-combinatorics}
Let $n, m \in \Z, n \geq 0$, and $q \in \k^{\times}$. The \emph{q-binomial coefficients} are defined recursively as follows.
First, set ${0 \choose 0}_q = 1$ and ${n \choose m}_q = 0$  if $m>n$ or $m<0$.
Then, for $n \geq 1$ and $0 \leq m \leq n$,
\begin{align}\label{def_q_binomial}
	{n \choose m}_q  =  {{n-1} \choose {m-1}}_q + q^{m}  {{n-1} \choose m}_q.
\end{align}

It follows that ${n \choose 0}_q = 1 = {n \choose n}_q$ and ${n \choose m}_q = {n \choose {n-m}}_q$.

The \emph{q-numbers} and the \emph{q-factorials} also are defined recursively: $(0)_q = 0$,  $(n)_q = \sum_{\ell=0}^{n-1}q^\ell;$ and 
$(0)_q ! = 1$, $(n)_q ! = (n)_q (n-1)_q !,$ respectively.

Note that, if $q \neq 1$, then $q$ is an $n^{th}$ root of unity if and only if $(n)_q = 0$.

Suppose that $n \geq 1$ and $q \in \k^{\times}$. If $(n-1)_q! \neq 0$, then 
\begin{align}\label{def_q_bino_fatorial}
	{n \choose m}_q  =  \dfrac{(n)_q!}{(n-m)_q!(m)_q!},
\end{align}
for all $0 < m < n$.
In particular, if $q \neq 1$ is a primitive $n^{th}$ root of unity, then ${n \choose m}_q  = 0,$
for all $0 < m < n$.

Suppose $q$ is a primitive $N^{th}$ root of unity.  For $k \geq 0$, let $k_D$ and $k_R$ be the integers determined by $k =k_D \ N + k_R$ where $0 \leq k_R <N$.
Then 
\begin{align}\label{radford}
	{n \choose m}_q  = {n_R \choose m_R}_q {n_D \choose m_D},
\end{align}
for all $0 \leq m \leq n$.

Furthermore, the following equalities are well-know: 
\begin{align}
	\sum_{k=0}^{m} (-1)^k {m \choose k}_q q^{\frac{k(k+1)}{2} - km}=0  \quad (m \geq 1),
\end{align}

\begin{align}\label{id_q}
	{n \choose m}_q = q^{m(n-m)} { n \choose m}_{q^{-1}},
\end{align}
\begin{align} \label{6.7}
	{n \choose m}_q  =  {{n-1} \choose m}_q  + q^{n-m}  {{n-1} \choose {m-1}}_q.
\end{align}

These definitions and equalities are classic for \emph{q-computations}, and one can find more details about them,  for instance, in \cite{radford}.
However, to compute partial actions of Taft algebras on algebras, we are going to need some more.
We present below \emph{q-identities} that will be used to compute such partial actions in the following section.

\begin{lem}\cite[Lemmas 3.2, 3.3 and 3.9]{taft_corpo_revista} \label{prodqbinom}  \label{lema1}  \label{lema4}
	Let $q \in \k^{\times}$ and $i,j,k,s, t \in \N_0$, such that $k \leq j$.
	Then, 
	\begin{align} \label{Lema2_2}
		&{ j \choose k }_q { {j-k} \choose {i-k} }_q = { j \choose i }_q { i \choose k }_q;\\ 
		\label{Lema2_6}
		&\sum_{\ell=0}^{j} (-1)^\ell q^{s \ell +\frac{\ell(\ell+1)}{2}} {j \choose \ell}_q {{j+t-\ell} \choose {j+s}}_q ={t \choose s}_q;\\
		\label{Lema2_10}
		&q^{s(i-j)} \sum_{\ell=0}^{j} {j \choose \ell}_q {{j+t-\ell} \choose {i+s-\ell}}_q {\ell \choose i}_q (-1)^{i-\ell} q^{\frac{(i-\ell)(i-\ell+1)}{2}}=	{j \choose i}_q {t \choose s}_q.
	\end{align}
\end{lem}

\begin{lem}
	Let $t,\ell \in \N_0$ and $q \in \k^{\times}$.
	Then,
	\begin{itemize}
		\item[(i)] if $0 \leq \ell < t$ and $(t-1)_q! \neq 0$, then
		\begin{equation}\label{igualdade_estrela_simetria} 
			\left(1-q^{t - \ell}\right) { t \choose \ell}_q  +q^{\ell+1} { t \choose \ell + 1}_q = { t \choose \ell + 1}_q;
		\end{equation}
		\item[(ii)] for each $m \in N_0$, it holds
		\begin{equation} 	\label{lema_acao_da_acao} 
			\sum_{s=0}^{m} (-1)^{m-s} q^{\frac{s(s+1)}{2}-sm} {m \choose s}_q {{s+t} \choose {\ell}}_q = q^{\frac{m(m+1)}{2} + m(t-\ell)} {t \choose \ell-m}_q.
		\end{equation}
	\end{itemize}
\end{lem}

\begin{proof}
	\begin{itemize}
		\item[(i)] It is straightforward using that $ { t \choose \ell + 1}_q = \frac{(t-\ell)_q}{(\ell + 1)_q}  { t \choose \ell}_q$  and \eqref{def_q_bino_fatorial}.
		
		\item[(ii)]	The proof is done by induction on $m$. The statement is clear for $m=0$. For $m > 0$, note that 
		\begin{align*}
			& \sum_{s=0}^{m+1} (-1)^{m+1-s} q^{\frac{s(s+1)}{2}-s(m+1)} {m+1 \choose s}_q {{s+t} \choose {\ell}}_q \\
			\stackrel{\eqref{def_q_binomial}}{=}  & \sum_{s=1}^{m+1} (-1)^{m+1-s} q^{\frac{s(s+1)}{2}-s(m+1)} {m \choose s-1}_q {{s+t} \choose {\ell}}_q \\	
			- & \sum_{s=0}^{m} (-1)^{m-s} q^{\frac{s(s+1)}{2}-sm} {m \choose s}_q {{s+t} \choose {\ell}}_q\\
			= & q^{-m}\sum_{s=0}^{m} (-1)^{m-s} q^{\frac{s(s+1)}{2}-sm} {m \choose s}_q {{s+(t+1)} \choose {\ell}}_q \\	
			- & \sum_{s=0}^{m} (-1)^{m-s} q^{\frac{s(s+1)}{2}-sm} {m \choose s}_q {{s+t} \choose {\ell}}_q.
		\end{align*}
		So, we conclude the induction step applying the induction hypothesis in both summations, and using equality \eqref{6.7}.
	\end{itemize}
\end{proof}

\begin{lem}
	Let $q\in \k \setminus\{1\}$ be a $n^{th}$ root of unity. Then,
	\begin{itemize}
		\item[(i)] for each $0 \leq i < n$, it holds that
		\begin{equation} \label{igualdade_q_n}
			(-1)^i q^{\frac{i(i+1)}{2}} { n-1 \choose i}_q = 1;
		\end{equation}
		\item[(ii)] if $q$ is a primitive root of unity, and $\ell, t \in \N_0$ such that $0 \leq \ell \leq t < n$, then it holds that 
		\begin{equation}\label{identidade_lema_simetria_n}
			(-1)^\ell q^{t \ell - \frac{\ell(\ell-1)}{2}} { n+\ell-t-1 \choose \ell}_q = { t \choose \ell}_q.
		\end{equation}
	\end{itemize}
\end{lem}

\begin{proof} Both items are proved by induction. For item (i) it can be used equality \eqref{def_q_binomial}, while for item (ii) it is a computation using \eqref{def_q_bino_fatorial}.
\end{proof}

\begin{lem}\label{Lema_R_L_simetria}
	Let $q\in \k$ be a primitive $n^{th}$ root of unity and $k, \ell, t \in \N_0$ such that $ k, \ell, t < n$. Then,
	\begin{itemize}
		\item[(i)] if $n \leq t+\ell$, then, for each $0 \leq m < t$, it holds that \begin{equation}\label{identidade_R_L}
			\begin{split}
				& \sum_{s=0}^{m}(-1)^s q^{\frac{s(s+1)}{2}-sm} { t \choose m-s}_q { s \choose t+\ell-n}_q \\
				= & \ (-1)^m q^{\frac{-m(m-1)}{2}-t(t+\ell-m)}  { m - t + n \choose \ell }_q;
			\end{split}
		\end{equation}
		\item[(ii)] if $\ell \leq k$, then it holds that
		\begin{equation}		
			\sum_{s=0}^{\ell}(-1)^s q^{\frac{s(s+1)}{2}-s(\ell-k)} { n+ \ell - k \choose s}_q = { k \choose \ell }_q;
		\end{equation}
		
		\item[(iii)] if $\ell \leq k$, then, for each $\ell \leq m < n$, it holds that
		\begin{equation}\label{identidade_R_L_item}
			\sum_{s=0}^{m}(-1)^s q^{\frac{s(s+1)}{2}-s(\ell-k)} { n+ m - k \choose s}_q { n + m - s \choose n+\ell-s}_q = { k \choose \ell }_q.
		\end{equation}
	\end{itemize}
\end{lem}
\begin{proof} We prove the statements by induction on the upper index of the summation symbol.
	\begin{itemize}
		\item[(i)] To prove this item, we consider two cases: $t+ \ell > n$ and $t + \ell =n$.
		
		First, assume $t+\ell>n$.
		For $m=0$, it is clear. 
		For the induction step, note that
		\begin{align*}
			& \sum_{s=0}^{m+1}(-1)^s q^{\frac{s(s+1)}{2}-s(m+1)} { t \choose m + 1-s}_q { s \choose t+\ell-n}_q \\
			= \ & \ \sum_{s=0}^{m}(-1)^{s+1} q^{\frac{(s+1)(s+2)}{2}-(s+1)(m+1)} { t \choose m -s}_q { s + 1 \choose t+\ell-n}_q \\
			\stackrel{\eqref{def_q_binomial}}{=} & \ -q^{-m}\sum_{s=0}^{m}(-1)^{s} q^{\frac{s(s+1)}{2}-sm} { t \choose m -s}_q { s  \choose t+\ell-n-1}_q \\
			- &\ q^{t+\ell-m}\sum_{s=0}^{m}(-1)^{s} q^{\frac{s(s+1)}{2}-sm} { t \choose m -s}_q { s \choose t+\ell-n}_q.
		\end{align*}
		Then, one conclude the induction step using the induction hypothesis in both summations and applying \eqref{def_q_binomial}.
		
		Now, assume $t+\ell=n$.
		Thus, identity \eqref{identidade_R_L} is rewritten as
		\begin{align*}
			\sum_{s=0}^{m}(-1)^s q^{\frac{s(s+1)}{2}-sm} { t \choose m-s}_q 
			= (-1)^m q^{mt - \frac{m(m-1)}{2}}  { m - t + n \choose m }_q .
		\end{align*}
		
		Again, the case $m=0$ is clear. For the induction step, assume $m \geq 0$ and note that 
		\begin{align*}
			& \sum_{s=0}^{m+1}(-1)^s q^{\frac{s(s+1)}{2}-s(m+1)} { t \choose m+1-s}_q \\
			= \ & \ { t \choose m+1} - q^{-m} \sum_{s=0}^{m}(-1)^{s} q^{\frac{s(s+1)}{2}-sm} { t \choose m-s}_q \\ 
			= \ & \ { t \choose m+1} - q^{-m} \left((-1)^m q^{mt- \frac{m(m-1)}{2}}  { m - t + n \choose m }_q \right), 
			\end{align*}
		where the induction hypothesis was used to obtain the last equality. Then, the result follows using equalities \eqref{identidade_lema_simetria_n} and \eqref{6.7}.
			
		\item[(ii)] The statement is clear for $\ell=0$. Assume $\ell > 0$ and $\ell+ 1 \leq k$. Then,  
		\begin{align*}
			& \sum_{s=0}^{\ell+1}(-1)^s q^{\frac{s(s+1)}{2}-s(\ell+1-k)} { n+ \ell +1 - k \choose s}_q \\
			\stackrel{\eqref{def_q_binomial}}{=} & 1+ \sum_{s=1}^{\ell+1}(-1)^s q^{\frac{s(s+1)}{2}-s(\ell+1-k)} { n+ \ell - k \choose s -1}_q \\
			+ & \sum_{s=1}^{\ell+1}(-1)^s q^{\frac{s(s+1)}{2}-s(\ell-k)} { n+ \ell - k \choose s}_q \\
			= & \sum_{s=0}^{\ell}(-1)^{s+1} q^{\frac{(s+1)(s+2)}{2}-(s+1)(\ell+1-k)} { n+ \ell - k \choose s}_q \\
			+ & \sum_{s=0}^{\ell}(-1)^s q^{\frac{s(s+1)}{2}-s(\ell-k)} { n+ \ell - k \choose s}_q \\
			+ &   (-1)^{\ell+1} q^{\frac{(\ell+1)(\ell+2)}{2}-(\ell+1)(\ell-k)} { n+ \ell - k \choose \ell + 1}_q\\
			=&  - q^{k - \ell} { k \choose \ell}_q + { k \choose \ell}_q   +q^{\ell+1} { k \choose \ell + 1}_q. 
		\end{align*}
		where the last equality follows from the induction hypothesis and equality \eqref{identidade_lema_simetria_n}. Hence, \eqref{igualdade_estrela_simetria} concludes the induction step.
		
		\item[(iii)] The statement holds for $m= \ell$ by the previous item.
		Now, for the induction hypothesis, note that
		\begin{align*}
			& \sum_{s=0}^{m+1}(-1)^s q^{\frac{s(s+1)}{2}-s(\ell-k)} { n+ m + 1 - k \choose s}_q { n + m + 1 - s \choose n+\ell-s}_q \\
			= & \sum_{s=0}^{m}(-1)^s q^{\frac{s(s+1)}{2}-s(\ell-k)} { n+ m + 1 - k \choose s}_q { n + m + 1 - s \choose n+\ell-s}_q,
		\end{align*} 
		since ${ n \choose n + \ell - m - 1}_q=0$, because $0<n+\ell -m -1 <n$.
		
		Then, by \eqref{def_q_binomial} we obtain the result.
	\end{itemize}
\end{proof}

\section{Partial (co)actions of the Taft Algebra}\label{sec:Taft}

\noindent{\bf Taft Algebra.}
Let $n \geq 2$ be an integer and suppose that $\k$ contains a primitive $n^{th}$ root of unity $q$.
In particular $char(\k) \nmid n$.
The \emph{Taft algebra of order $n$}, or shortly \emph{Taft algebra}, here denoted by $T_n(q)$, has the following structure:
as algebra it is generated over $\k$ by two elements $g$ and $x$ with relations $g^n=1,$ $x^n =0$ and $xg = q gx$.
Thus, the set $\{g^ix^j \ : \ 0 \leq   i, j <n \}$ is the canonical basis for $T_n(q)$ and so $dim_\k(T_n(q)) = n^{2}$.
The coalgebra structure of $T_n(q)$ is induced by setting $g$ a group-like element and $x$ an $(1,g)$-primitive element, that is, $\Delta (g) = g \o g$, $\Delta (x) = x \o 1 + g \o x$, $\varepsilon (g) = 1$ and $\varepsilon (x) = 0$.
In general, the comultiplication map is $\Delta (g^ix^j) = \sum_{\ell=0}^j {j \choose \ell}_q  g^{i+\ell}x^{j-\ell} \o g^ix^\ell$, for $i, j \in \I_{0, n-1}$.
To complete the Hopf algebra structure of $T_n(q)$, the antipode map $S$ is defined by $S(g) =g^{n-1}$ and $S(x) = -g^{n-1}x.$
Furthermore, note that the group of group-like elements is $G( T_n(q) ) = \{ 1, g, \cdots, g^{n-1} \}$, \emph{i.e.}, $ G( T_n(q) ) =  C_n.$

\subsection{Partial actions} 
In this section, we characterize suitable partial actions of $T_n(q)$ on an algebra $A$. If $A$ has only trivial idempotent elements, these are all partial actions.

\begin{lem}\label{idempotente_comuta}
	Let $ \cdot : H \otimes A \longrightarrow A$ be a partial action of $H$ on $A$. For each $h \in G(H)$ and $a \in A$:
	\begin{itemize}
		\item[(i)]  $(h \cdot 1_A)a(h \cdot 1_A) = (h \cdot 1_A)a$;
		\item[(ii)] if the partial action $\cdot|_{\Bbbk G(H)\otimes A}$ is symmetric, then $(h \cdot 1_A)a=a(h \cdot 1_A)$;
		\item[(iii)] if $h \cdot 1_A =0$, then $(h^i \cdot 1_A)(h^{i+1} \cdot 1_A)=0$, for all $i \in \Z$.
	\end{itemize}
\end{lem}
\begin{proof}
	Assume  $h \in G(H)$ and $a \in A$.
	\begin{itemize}
		\item[(i)] First,
	$
	(h \cdot 1_A)(h \cdot (h^{-1} \cdot a)) \stackrel{(PA.2)}{=} (h \cdot 1_A)(h \cdot 1_A)(1_H \cdot a) = (h \cdot 1_A)a. 
	$
	Then, \begin{align*}
	(h \cdot 1_A)a(h \cdot 1_A) & =  (h \cdot 1_A)(h \cdot (h^{-1} \cdot a))(h \cdot 1_A) \\ 
	&\!\!\!\!\! \stackrel{(PA.2)}{=}  (h \cdot 1_A)(h \cdot ((h^{-1} \cdot a) 1_A)) =  (h \cdot 1_A)a.
		\end{align*}

\item[(ii)]	Analogously to the previous item.

\item[(iii)] For each $i \in \Z$,  $(h^i \cdot 1_A)(h^{i+1} \cdot 1_A) \stackrel{(PA.3)}{=} h^i \cdot (h \cdot 1_A) =  h^i \cdot 0 = 0.$ 
\end{itemize}
\end{proof}

\begin{prop}\label{parte_global}
	Let $\cdot : T_n(q) \o A \longrightarrow A$ be a partial action of $T_n(q)$ on $A$.
	If $g \cdot 1_A = 1_A$, then $\cdot$ is a global action.
\end{prop}

\begin{proof}
	Suppose that $g \cdot 1_A = 1_A$.
	It is enough to prove that $g \cdot (h \cdot a) = gh \cdot a$ and $x \cdot (h \cdot a) = xh \cdot a$, for any $h \in T_n(q)$ and $a \in A$.
	First, by (PA.2), we get
	$$x \cdot 1_A = x \cdot 1_A^2 =  (x \cdot 1_A)(1 \cdot 1_A) + (g \cdot 1_A)(x \cdot 1_A)= 2 (x \cdot 1_A),$$
	that is, $x \cdot 1_A = 0$.	Then, by (PA.3),
	$$x \cdot (h \cdot a) = (x \cdot 1_A)(h \cdot a) + (g \cdot 1_A)(xh \cdot a) = xh \cdot a.$$
	Also by (PA.3), it follows that 
	$g \cdot (h \cdot a) = (g \cdot 1_A)(gh \cdot a) = gh \cdot a.$
\end{proof}

\begin{prop}\label{formula_eh_acao_parcial}
	Let $\cdot : T_n(q) \o A \longrightarrow A$ be a partial action of $T_n(q)$ on $A$.
	If $g \cdot 1_A = 0$, then
	 \begin{equation}\label{formula_3_1}
	 g^ix^j \cdot a = q^{-ij} \sum_{k=0}^j (-1)^k q^{- \frac{k(k-1)}{2}} { j \choose k }_{q^{-1}} (x \cdot 1_A)^{j-k}(g^{i+k} \cdot a)(x \cdot 1_A)^k,
	  \end{equation}	 
	for any $a \in A$ and $i,j \in\I_{0, n-1}.$
\end{prop}

\begin{proof}
	First, we note that $g^{-1}x \cdot a = -q a (x \cdot 1_A)$.
	Indeed, by (PA.1) and (PA.3), we get $g^{-1} \cdot (g \cdot 1_A) = (g^{-1} \cdot 1_A)(1 \cdot 1_A)$, that is, $g^{-1} \cdot 1_A =0$. 
	Again, by (PA.1) and (PA.3), we conclude that
	$$ g^{-1}x \cdot (g \cdot 1_A) = (g^{-1}x \cdot 1_A)(1 \cdot 1_A) + (1 \cdot 1_A)(g^{-1}xg \cdot 1_A) = (g^{-1}x \cdot 1_A) + q (x \cdot 1_A),$$
	that is, $g^{-1}x \cdot 1_A = -q (x \cdot 1_A).$ 
	Then, by (PA.1) and (PA.2), we obtain
	$$g^{-1}x \cdot a = (g^{-1}x \cdot a)(g^{-1} \cdot 1_A) + (1 \cdot a)(g^{-1}x \cdot 1_A) = -q a (x \cdot 1_A).$$  
	
	Now, we prove the desired equality by induction on $j$. It is immediate for $j=0$, and if one wants to check for $j=1$ it follows by (PA.3) in $g^{-1}x \cdot (g^{i+1} \cdot a)$. Assume $j \geq 0$.
	Then, by (PA.3), we get
	\begin{align*}
	g^{-1}x \cdot (g^{i+1}x^j \cdot a) = (g^{-1}x \cdot 1_A)(g^{-1}g^{i+1}x^j \cdot a)+ (1 \cdot 1_A)(g^{-1}xg^{i+1}x^{j} \cdot a),
	\end{align*}
	that is, $-q (g^{i+1}x^j \cdot a)(x \cdot 1_A) = -q (x \cdot 1_A)(g^ix^j \cdot a) + q^{i+1} (g^ix^{j+1} \cdot a),$
	and thus $q^i\left(g^ix^{j+1} \cdot a \right) = (x \cdot 1_A)(g^{i}x^j \cdot a)-( g^{i+1}x^j \cdot a)(x \cdot 1_A).$
	
	Using the induction hypothesis, we obtain
	\begin{align*}
	& q^i (g^{i}x^{j+1} \cdot a ) \\
	= & (x \cdot 1_A) \left( q^{-ij} \sum_{k=0}^j (-1)^k q^{- \frac{k(k-1)}{2}} { j \choose k }_{q^{-1}} (x \cdot 1_A)^{j-k}(g^{i+k} \cdot a)(x \cdot 1_A)^k \right) \\
	- & \left(\!\! q^{-(i+1)j} \sum_{k=0}^j (-1)^k q^{- \frac{k(k-1)}{2}} { j \choose k }_{q^{-1}} \!\!\!\!\!\!\!\!\!(x \cdot 1_A)^{j-k}(g^{i+1+k} \cdot a)(x \cdot 1_A)^k \!\!\right)\!\!(x \cdot 1_A) \\
	= & \ q^{-ij}\Big( (x \cdot 1_A)^{j+1}(g^i \cdot 1_A) \\
	+ & \sum_{k=1}^j (-1)^k q^{- \frac{k(k-1)}{2}} { j \choose k }_{q^{-1}} (x \cdot 1_A)^{j+1-k}(g^{i+k} \cdot a)(x \cdot 1_A)^k \\
	- & q^{-j} \sum_{k=0}^{j-1} (-1)^k q^{- \frac{k(k-1)}{2}} { j \choose k }_{q^{-1}} (x \cdot 1_A)^{j-k}(g^{i+1+k} \cdot a)(x \cdot 1_A)^{k+1} \\
	- & q^{-j} (-1)^j q^{-\frac{j(j-1)}{2}} (g^{i+1+j} \cdot a) (x \cdot 1_A)^{j+1} \Big) \\
	= & \ q^{-ij}\Big( (x \cdot 1_A)^{j+1}(g^i \cdot 1_A) \\
	+ & \sum_{k=1}^j (-1)^k q^{- \frac{k(k-1)}{2}} \left({ j \choose k }_{q^{-1}} + q^{-j+(k-1)} { j \choose k-1 }_{q^{-1}} \right) \\
	\times & (x \cdot 1_A)^{j+1-k}(g^{i+k} \cdot a)(x \cdot 1_A)^k + (-1)^{j+1} q^{-\frac{j(j+1)}{2}} (g^{i+1+j} \cdot a) (x \cdot 1_A)^{j+1} \Big) \\
	\stackrel{\eqref{6.7}}{=} & q^{-ij}\Big( (x \cdot 1_A)^{j+1}(g^i \cdot 1_A) \\
	+ & \sum_{k=1}^j (-1)^k q^{- \frac{k(k-1)}{2}} { j + 1 \choose k }_{q^{-1}} (x \cdot 1_A)^{j+1-k}(g^{i+k} \cdot a)(x \cdot 1_A)^k \\
	+ & (-1)^{j+1} q^{-\frac{j(j+1)}{2}} (g^{i+1+j} \cdot a) (x \cdot 1_A)^{j+1} \Big)
	\end{align*}
	that is,
	\begin{align*}
	& g^{i}x^{j+1} \cdot 1_A \\
	= &  q^{-i(j+1)}\sum_{k=0}^{j+1} (-1)^k q^{- \frac{k(k-1)}{2}} { j + 1 \choose k }_{q^{-1}} (x \cdot 1_A)^{j+1-k}(g^{i+k} \cdot a)(x \cdot 1_A)^k.
	\end{align*}
	It concludes the induction step.
\end{proof}

\begin{cor}\label{corolario_ida}
	Let $\cdot : T_n(q) \o A \longrightarrow A$ be a partial action of $T_n(q)$ on $A$.
	If $g \cdot 1_A = 0$, then
	\begin{itemize}
		\item[(i)] $(x\cdot 1_A)^n\in Z(A)$;
		\item[(ii)] $g^i \cdot (x\cdot 1_A) =q^{-i}(g^i \cdot 1_A)(x\cdot 1_A)$, for all $ i\in\I_{0, n-1}$;
		\item[(iii)] If there is $ i\in\I_{2, n-1}$ such that $g^i\cdot a=a$, for all $a \in A$, then $x\cdot 1_A=0.$
	\end{itemize} 
\end{cor}

\begin{proof}
\begin{itemize}
			\item[(i)] By definition of the partial action, we have
			$$g^{n-1}x\cdot(gx^{n-1}\cdot a)=(g^{n-1}x\cdot 1_A)(x^{n-1}\cdot a)$$
			As seen in the proof of Proposition \ref{formula_eh_acao_parcial}, it holds that $g^{n-1}x\cdot a=-qa(x\cdot 1_A)$. Then, the above equality can be rewritten as 
			$$(gx^{n-1}\cdot a)(x\cdot 1_A)=(x\cdot 1_A)(x^{n-1}\cdot a).$$
			Then, by \eqref{formula_3_1}, we obtain
			\begin{eqnarray*}
				&&	q\sum_{k=0}^{n-1} (-1)^k q^{- \frac{k(k-1)}{2}} { n-1 \choose k }_{q^{-1}} (x \cdot 1_A)^{n-1-k}(g^{k+1} \cdot a)(x \cdot 1_A)^{k+1}\\
				&=& \sum_{k=0}^{n-1} (-1)^k q^{- \frac{k(k-1)}{2}} { n-1 \choose k }_{q^{-1}} (x \cdot 1_A)^{n-k}(g^{k} \cdot a)(x \cdot 1_A)^k.
			\end{eqnarray*}
			Now, since $g \cdot 1_A =0$, and so $g \cdot a= g^{-1} \cdot a=0$, for all $a \in A$, 
			\begin{eqnarray*}
				&&\sum_{k=1}^{n-3} (-1)^k q^{- \frac{k(k-1)}{2}} q { n-1 \choose k }_{q^{-1}}  (x \cdot 1_A)^{n-(k+1)}(g^{k+1} \cdot a)(x \cdot 1_A)^{k+1}\\
				&+& (-1)^{n-1} q^{- \frac{(n-1)(n-2)}{2}} q { n-1 \choose n-1 }_{q^{-1}} a(x\cdot 1_A)^n\\
				&=& (x\cdot 1_A)^n a+\sum_{k=2}^{n-2} (-1)^k q^{- \frac{k(k-1)}{2}} { n-1 \choose k }_{q^{-1}} (x \cdot 1_A)^{n-k}(g^{k} \cdot a)(x \cdot 1_A)^k.
			\end{eqnarray*}	 
		Then, with a suitable change in the summation index,
		\begin{eqnarray*}
				&&\sum_{k=2}^{n-2} -(-1)^k q^{- \frac{k(k-1)}{2}} q^k { n-1 \choose k-1 }_{q^{-1}}  (x \cdot 1_A)^{n-k}(g^{k} \cdot a)(x \cdot 1_A)^{k}\\
				&+& (-1)^{n-1} q^{- \frac{n(n-1)}{2}} { n-1 \choose n-1 }_{q^{-1}} a(x\cdot 1_A)^n\\
				&=&  (x\cdot 1_A)^n a+\sum_{k=2}^{n-2} (-1)^k q^{- \frac{k(k-1)}{2}} { n-1 \choose k }_{q^{-1}} (x \cdot 1_A)^{n-k}(g^{k} \cdot a)(x \cdot 1_A)^k.
			\end{eqnarray*}
		At last, using \eqref{igualdade_q_n} and ${ n-1 \choose k }_{q^{-1}}=-q^k{ n-1 \choose k-1 }_{q^{-1}}$, we conclude that
			$a(x\cdot 1_A)^n=(x\cdot 1_A)^na$, for all $a\in A$.
			\item[(ii)] 
			Since $g^i \in G(T_n(q))$ and $g \cdot 1_A=0$, we have
			\begin{align*}
			g^i \cdot (x\cdot 1_A) = & \ (g^i\cdot 1_A)(g^ix\cdot 1_A) \\ \stackrel{\eqref{formula_3_1}}{=} & \ q^{-i}(g^i \cdot 1_A)((x\cdot 1_A)(g^i \cdot 1_A)-(g^{i+1} \cdot 1_A)(x\cdot 1_A)) \\
		= & \ q^{-i}((g^i \cdot 1_A)(x\cdot 1_A)(g^i \cdot 1_A)-(g^i \cdot 1_A)(g^{i+1} \cdot 1_A)(x\cdot 1_A))\\
			\stackrel{ \ref{idempotente_comuta} (iii)}{=}  & \ q^{-i}(g^i \cdot 1_A)(x\cdot 1_A)(g^i \cdot 1_A) \\
			\stackrel{ \ref{idempotente_comuta}(i)}{=} & q^{-i}(g^i \cdot 1_A)(x\cdot 1_A).
			\end{align*}
	\item[(iii)]  In that case, by the previous item, we conclude that $(x\cdot 1_A)=q^{-i}(x\cdot 1_A)$, for some $i\in\I_{2, n-1}$, and therefore $x\cdot 1_A=0$.
		\end{itemize}
\end{proof}

\begin{lem}\label{lema_ind}
	Let $\cdot : \k C_n \otimes A \longrightarrow A$ be a partial action of $\k C_n$ on $A$.
	Then, for each $w \in A$ such that
	$g^i \cdot w =q^{-i}(g^i \cdot 1_A)w$,
	for all $ i\in\I_{0, n-1}$, the identity
	$$ ( g^i \cdot w)^\ell = q^{-i\ell}(g^i \cdot 1_A)w^\ell$$
	holds for all $\ell \geq 1$. Moreover, if the partial action is symmetric, 
	then $$ ( g^i \cdot w)^\ell = q^{-i\ell}w^\ell(g^i \cdot 1_A).$$
\end{lem}
\begin{proof}
	The result follows by induction on $\ell$ 	using Lemma \ref{idempotente_comuta} (i).
\end{proof}

From the above lemma we have the following consequence.

\begin{cor}\label{cor_condw}
		Let $\cdot : \k C_n \otimes A \longrightarrow A$ be a partial action of $\k C_n$ on $A$. Then, for each $w \in A$ such that
$g^i \cdot w =q^{-i}(g^i \cdot 1_A)w$,
	for all $i,j \in \I_{0,n-1}$, the identity
	$$ g^i \cdot (w^\ell (g^j \cdot a) w^k) = q^{-i(\ell + k)}(g^i \cdot 1_A)w^\ell(g^{i+j} \cdot a) w^k$$
	holds for all $a \in A$, $\ell, k \geq 0$. 
In particular, $ g^i \cdot w^k = q^{-ik}(g^i \cdot 1) w^k$.
	
	Moreover, if the partial action is symmetric, 	then $$ g^i \cdot (w^\ell (g^j \cdot a) w^k) = q^{-i(\ell + k)}w^\ell(g^{i+j} \cdot a) w^k(g^i \cdot 1_A).$$	
		\end{cor}
\begin{proof}
First, since $g^i$ is a group-like element, then $g^i \cdot a^{\ell} = (g^i \cdot a)^{\ell}$ for all $a \in A$. Now, notice that for each $w \in A$ such that $g^i \cdot w =q^{-i}(g^i \cdot 1_A)w$,
\begin{eqnarray*}
 g^i \cdot (w^\ell (g^j \cdot a) w^k)&=& (g^i \cdot w)^{\ell} (g^i \cdot 1_A)(g^{i+j} \cdot a)(g^i \cdot w)^{k}\\ 
 &\stackrel{\ref{lema_ind}}{=} &  q^{-i\ell}(g^i \cdot 1_A) w^{\ell}(g^i \cdot 1_A)(g^{i+j} \cdot a) q^{-ik}(g^i \cdot 1_A) w^{k}\\
 &\stackrel{\ref{idempotente_comuta}(i)}{=}& q^{-i(\ell + k)} (g^i \cdot 1_A) w^\ell(g^{i+j} \cdot a) w^k,
\end{eqnarray*}
for all $a \in A$, $\ell, k \geq 0$.

The symmetric part is analogous.
\end{proof}

\begin{prop}\label{prop_volta}
	Let $\cdot : \k C_n \otimes A \longrightarrow A$ be a partial action of $\k C_n$ on $A$, and suppose $g \cdot 1_A = 0$.
	Then, for each $w \in A$ such that $w^n \in Z(A)$ and $g^i \cdot w = q^{-i}(g^i \cdot 1_A)w$, for all $i \in \I_{0,n-1}$, the linear map $\cdot : T_n(q) \otimes A \longrightarrow A$,  defined as 
	\begin{align}\label{formula}
	g^ix^j \cdot a = q^{-ij} \sum_{k=0}^j (-1)^k q^{- \frac{k(k-1)}{2}} { j \choose k }_{q^{-1}} w^{j-k}(g^{i+k} \cdot a) w^k,
	\end{align}
	is a partial action of $T_n(q)$ on $A$.
\end{prop}
\begin{proof}	Since (PA.1) clearly holds, it is sufficient to check condition (PA.2$^{\prime}$):
$$g^ix^j \cdot (a(g^t x^s \cdot b)) = ((g^ix^j)_1 \cdot a)((g^ix^j)_2 g^t x^s \cdot b).$$
We will proceed by fixing the vectors and checking the scalars. On one hand, 
\begin{align*}
& g^ix^j \cdot (a(g^tx^s \cdot b)) =  q^{-ij} \sum_{k=0}^j (-1)^{k} q^{\frac{-k(k-1)}{2}} {j \choose k}_{q^{-1}} w^{j-k}  g^{i+k}\cdot(a(g^tx^s \cdot b)) w^{k}\\
& = q^{-ij-st} \sum_{k=0}^j\sum_{\ell=0}^s (-1)^{k+\ell} q^{- \frac{k(k-1)}{2}}q^{\frac{-\ell(\ell-1)}{2}} {j \choose k}_{q^{-1}}{s \choose \ell }_{q^{-1}} w^{j-k} (g^{i+k}\cdot a)\\ 
 & \times (g^{i+k}\cdot(w^{s-\ell}(g^{t+\ell} \cdot b) w^\ell))w^{k}\\
 & \stackrel{\ref{cor_condw}}{=} q^{-ij-s(t+i)} \sum_{k=0}^j\sum_{\ell=0}^s (-1)^{k+\ell} q^{-sk} q^{- \frac{k(k-1)}{2}} q^{\frac{-\ell(\ell-1)}{2}} {j \choose k}_{q^{-1}}{s \choose \ell }_{q^{-1}}\\ 
 & \times w^{j-k} (g^{i+k}\cdot a) w^{s-\ell}(g^{i+t+k+\ell} \cdot b) w^{k+\ell}.
\end{align*}

Let $ k_0 \in\I_{0, j}$ and $ \ell_0\in\I_{0, s}$ be fixed elements, thus we obtain the vector  
\begin{equation}\label{eq_vetor}
w^{j-k_0} (g^{i+k_0} \cdot a) w^{s-\ell_0} (g^{i+t+k_0+\ell_0} \cdot b) w^{k_0+\ell_0}
\end{equation} with scalar
\begin{equation}\label{eq_lado1}
q^{-ij-s(t+i)} \left( (-1)^{k_0+\ell_0} q^{-sk_0} q^{- \frac{k_0(k_0-1)}{2}} q^{\frac{-\ell_0(\ell_0-1)}{2}} {j \choose k_0}_{q^{-1}}{s \choose \ell_0}_{q^{-1}} \right).
\end{equation}
On the other hand, 
\begin{align*}
& ((g^ix^j)_1 \cdot a)((g^ix^j)_2g^tx^s \cdot b) 
=  \  \sum_{\ell =0}^j {j \choose \ell}_q  (g^{i+\ell}x^{j-\ell} \cdot a) ( g^ix^\ell g^t x^s \cdot b) \\
\stackrel{\ref{id_q}}{=} &  \  \sum_{\ell=0}^j q^{\ell(j-\ell)} q^{\ell t} {j \choose \ell}_{q^{-1}} (g^{i+\ell}x^{j-\ell} \cdot a) ( g^{i+t}x^{\ell + s} \cdot b).
\end{align*}

Due to term $g^{i+t}x^{\ell+s} \cdot b$, we have to proceed by cases:

\underline{Case 1: $j+s < n$.}
\begin{align*}
& \sum_{\ell=0}^j q^{\ell(j-\ell)} q^{\ell t} {j \choose \ell}_{q^{-1}} (g^{i+\ell}x^{j-\ell} \cdot a) ( g^{i+t}x^{\ell + s} \cdot b)\\
= & \sum_{\ell=0}^j  \sum_{r=0}^{j-l} \sum_{\theta=0}^{\ell+s} q^{-(i+\ell)(j-\ell)} q^{-(i+t)(\ell+s)}q^{\ell(j-\ell)} q^{\ell t} (-1)^{r+\theta} q^{\frac{-r(r-1)}{2}} q^{\frac{-\theta(\theta-1)}{2}} \\ 
\times & {j \choose \ell}_{q^{-1}} \!\!{j-\ell \choose r}_{q^{-1}} \!\!{\ell+s \choose \theta}_{q^{-1}} \!\!\!\!\!\!\! w^{j-(\ell+r)}  (g^{i+\ell+r}\cdot a) w^{r+\ell+s-\theta} (g^{i+t+\theta}\cdot b)w^{\theta}\\
= & \ q^{-ij-s(t+i)} \sum_{\ell=0}^j  \sum_{r=0}^{j-l} \sum_{\theta=0}^{\ell+s}(-1)^{r+\theta} q^{\frac{-r(r-1)}{2}} q^{\frac{-\theta(\theta-1)}{2}} {j \choose \ell}_{q^{-1}} {j-\ell \choose r}_{q^{-1}} \\ 
\times & {\ell+s \choose \theta}_{q^{-1}} w^{j-(\ell+r)}  (g^{i+\ell+r}\cdot a) w^{r+\ell+s-\theta} (g^{i+t+\theta}\cdot b)w^{\theta}\\
\stackrel{k=\ell+r}{=} & q^{-ij-s(t+i)} \sum_{\ell=0}^j  \sum_{k=\ell}^{j} \sum_{\theta=0}^{\ell+s}(-1)^{k-\ell+\theta} q^{\frac{-(k-\ell)(k-\ell-1)}{2}} q^{\frac{-\theta(\theta-1)}{2}} {j \choose \ell}_{q^{-1}}  \\
\times & {j-\ell \choose k-\ell}_{q^{-1}} {\ell+s \choose \theta}_{q^{-1}} w^{j-k}  (g^{i+k}\cdot a) w^{k+s-\theta} (g^{i+t+\theta}\cdot b)w^{\theta}.
\end{align*}

Setting $0\leq k_0\leq j$, we have that each $0 \leq \ell \leq j$ determines the vector in \eqref{eq_vetor}, considering $k=k_0$ and $\theta=k_0+\ell_0$, with scalar
\begin{align*}
& q^{-ij-s(t+i)} \left(\sum_{\ell=0}^{k_0}  (-1)^{\ell_0-\ell} q^{\frac{-(k_0-\ell)(k_0-\ell-1)}{2}} q^{\frac{-(k_0+\ell_0)(k_0+\ell_0-1)}{2}} \right. \\
\times & \left. {j \choose \ell}_{q^{-1}} {j-\ell \choose k_0-\ell}_{q^{-1}} {\ell+s \choose k_0+\ell_0}_{q^{-1}} \right) \\
\stackrel{\ref{prodqbinom}}{=} & \ \ q^{-ij-s(t+i)} q^{\frac{-k_0(k_0-1)}{2}} q^{\frac{-\ell_0(\ell_0-1)}{2}} (-1)^{\ell_0} {j \choose k_0}_{q^{-1}}\\
\times & \left (\sum_{\ell=0}^{k_0}  (-1)^{\ell} q^{\frac{-k_0(k_0-1)}{2}}q^{\frac{-\ell(\ell-1)}{2}} q^{\frac{-k_0(k_0-1)}{2}} q^{\ell k_0} q^{-\ell_0k_0}  {k_0 \choose \ell}_{q^{-1}} {\ell+s \choose k_0+\ell_0}_{q^{-1}} \right).  
\end{align*}

Therefore, looking to the highlighted scalar in \eqref{eq_lado1}, we just need to see that
\begin{align*}	
& \sum_{\ell=0}^{k_0}  (-1)^{k_0-\ell} q^{\frac{-\ell(\ell-1)}{2}+\ell k_0} q^{\frac{-k_0(k_0-1)}{2}}  q^{-\ell_0k_0}  {k_0 \choose \ell}_{q^{-1}} {\ell+s \choose k_0+\ell_0}_{q^{-1}} \\
=& \  q^{-sk_0} {j \choose k_0}_{q^{-1}}{s \choose \ell_0}_{q^{-1}}, 
\end{align*}
but it holds by Lemma  \ref{lema_acao_da_acao}.

\underline{Case 2: $j+s \geq n$.}
\begin{align*}
&    \sum_{\ell=0}^j q^{\ell(j-\ell)} q^{\ell t} {j \choose \ell}_{q^{-1}} (g^{i+\ell}x^{j-\ell} \cdot a) ( g^{i+t}x^{\ell + s} \cdot b)\\
= & \ q^{-ij-s(t+i)} \sum_{\ell =0}^{n-s-1} \sum_{r=\ell}^{j} \sum_{\theta=0}^{r+s} (-1)^{r - \ell+ \theta} q^{\frac{-(r - \ell)(r - \ell-1)}{2}-\frac{\theta(\theta-1)}{2}} {j \choose \ell}_{q^{-1}} {j-\ell \choose r - \ell}_{q^{-1}} \\
\times & {\ell+ s \choose \theta}_{q^{-1}} w^{j- r} (g^{i+r} \cdot a) w^{r+s-\theta} (g^{i+t+\theta} \cdot b) w^{\theta}.
\end{align*}

Note that, since $0 \leq \ell \leq n-s-1$, we can change the upper bound of the summation in $\theta$ for $n-1$ instead $r+s$, due to ${\ell+ s \choose \theta}_{q^{-1}}$:
\begin{align*}
& q^{-ij-s(t+i)} \sum_{\ell =0}^{n-s-1} \sum_{r=\ell}^{j} \sum_{\theta=0}^{n-1} (-1)^{r - \ell+ \theta} q^{\frac{-(r - \ell)(r - \ell-1)}{2}-\frac{\theta(\theta-1)}{2}} {j \choose \ell}_{q^{-1}} {j-\ell \choose r - \ell}_{q^{-1}} \\
\times & {\ell+ s \choose \theta}_{q^{-1}} w^{j- r} (g^{i+r} \cdot a) w^{r+s-\theta} (g^{i+t+\theta} \cdot b) w^{\theta}.
\end{align*}

Now, we will investigate when the vector highlighted in \eqref{eq_vetor} appears in the above expression. We can assume $k_0+\ell_0 \geq n$, since $k_0+\ell_0 < n$ is included in Case 1. Thus, $k_0+s \geq k_0 + \ell_0 \geq n$, what leads to $0 \leq n-s-1 < k_0$. So, to obtain the referred vector, there is a unique choice for $r$, and so only one for $\theta$ too, namely $r=k_0$ and $\theta= k_0+\ell_0-n$ (where the latter uses \eqref{radford}). Hence, 
\begin{align*}
& q^{-ij-s(t+i)} \sum_{\ell =0}^{n-s-1}  (-1)^{\ell_0-n-\ell} q^{\frac{-(k_0 - \ell)(k_0 - \ell-1)}{2}-\frac{(k_0 + \ell_0-n)(k_0 + \ell_0-n-1)}{2}} \\
\times & {j \choose \ell}_{q^{-1}} {j-\ell \choose k_0 - \ell}_{q^{-1}} {\ell+ s \choose k_0 + \ell_0-n}_{q^{-1}}\\
\times & \ w^{j- k_0} (g^{i+k_0} \cdot a)  w^{s-\ell_0+n} (g^{i+t+k_0+\ell_0-n} \cdot b) w^{k_0+\ell_0-n}.
\end{align*}

The vector in \eqref{eq_vetor} and the one in the above expression are the same due the commutative of $w^n$. Therefore, comparing the above scalar with the one highlighted in \eqref{eq_lado1}, we want to see that the following equality is true
\begin{align*}
& (-1)^{k_0+\ell_0} q^{-sk_0} q^{\frac{-k_0(k_0-1)}{2}-\frac{\ell_0(\ell_0-1)}{2}} {j \choose k_0}_{q^{-1}} {s \choose \ell_0}_{q^{-1}} = \sum_{\ell =0}^{n-s-1}  (-1)^{\ell_0-n-\ell}\\
 \times & q^{\frac{-(k_0 - \ell)(k_0 - \ell-1)}{2}-\frac{(k_0 + \ell_0-n)(k_0 + \ell_0-n-1)}{2}} {j \choose \ell}_{q^{-1}} {j-\ell \choose k_0 - \ell}_{q^{-1}} {\ell+ s \choose k_0 + \ell_0-n}_{q^{-1}}.
\end{align*}

But, applying identities \eqref{Lema2_2}, \eqref{igualdade_q_n} and \eqref{lema_acao_da_acao} to the above expression, it is reduced to
\begin{equation}\label{eq_final}
q^{\frac{k_0(k_0-1)}{2}+k_0(\ell_0-s)}\! {s \choose \ell_
	0}_{q^{-1}} \!\!\!\!=\!\!\!\sum_{\ell =n-s}^{k_0}  \!\!\!(-1)^{k_0-\ell} q^{\frac{-\ell(\ell+1)}{2}+\ell k_0} {k_0 \choose \ell}_{q^{-1}} \!\!{\ell+ s-n \choose k_0 + \ell_0-n}_{q^{-1}}.
\end{equation}

Since $n-s \leq \ell \leq k_0$ and $0 \leq \ell-n+s \leq k_0-n+s$, the right side in above expression can be rewritten replacing the variable $\ell$ to $L+n-s$ as
\begin{align*}
\sum_{L =0}^{k_0-n+s} \!\! (-1)^{k_0-L-n+s} q^{\frac{-(L+n-s)(L+n-s+1-2k_0)}{2}} {k_0 \choose L+n-s}_{q^{-1}} {L \choose k_0 + \ell_0-n}_{q^{-1}}\\
\stackrel{\eqref{igualdade_q_n}}{=} \!\!\sum_{L =0}^{k_0-n+s} \!\! (-1)^{k_0-L+s+1} q^{\frac{-(L-s)(L-s+1-2k_0)}{2}} {k_0 \choose L+n-s}_{q^{-1}} {L \choose k_0 + \ell_0-n}_{q^{-1}}.
\end{align*}

Therefore, \eqref{eq_final} follows from identity \eqref{identidade_R_L}.
\end{proof}

\begin{rem}
	Notice, in Proposition \ref{prop_volta}, that $x\cdot 1_A=w$, $x\cdot a=wa$, and $g^ix^j \cdot a$ is determined by the partial group action and this fixed element $w$. Hence, a partial action $T_n(q)$ on an algebra $A$ depends only on a partial group action, an element $w \in A$, and a compatibility relation between them.
\end{rem}

\begin{thm}\label{taft_principal}
	Let $\cdot:T_n(q) \o A \longrightarrow A$ be a linear map, and suppose $g \cdot 1_A \in \{0, 1_A\}$.
	Then, $\cdot$ is a partial action of $T_n(q)$ on $A$ if and only if  $\cdot$ is a global action, or $\cdot$ is a linear map such that $\cdot|_{\k C_n \otimes A}$
	 is a partial group action such that $g \cdot 1_A = 0$ and
	\begin{itemize}
		\item[(i)] $ g^ix^j \cdot a = q^{-ij} \sum_{k=0}^j (-1)^k q^{- \frac{k(k-1)}{2}} { j \choose k }_{q^{-1}} (x \cdot 1_A)^{j-k}(g^{i+k} \cdot a)(x \cdot 1_A)^k;$
		\item[(ii)]  $(x \cdot 1_A)^n \in Z(A)$; 
		\item[(iii)] $g^i \cdot (x \cdot 1_A) = q^{-i}(g^i \cdot 1_A)(x \cdot 1_A)$,
	\end{itemize}		
		for all $0 \leq i,j <n, a \in A$.
\end{thm}

	\begin{proof}
One part follows by Propositions \ref{parte_global}, \ref{formula_eh_acao_parcial} and Corollary \ref{corolario_ida}, and the converse follows by Proposition \ref{prop_volta} .
\end{proof}

\begin{cor}\label{idempotente_trivial}
	Suppose that $A$ is an algebra that has only the trivial idempotents. Then, Theorem \ref{taft_principal} characterizes all partial actions of $T_n(q)$ on A.
\end{cor}

	From the above result we have a complete characterization for partial actions of $T_n(q)$ on some families of algebras, such as division algebras, the tensorial algebra, the Weyl algebra and the universal enveloping algebra of a Lie algebra. 
	All the partial actions of $T_n(q)$ on such algebras are described in Theorem \ref{taft_principal} . 

\begin{cor}
		Let $\cdot:T_n(q) \o A \longrightarrow A$ be a linear map, and suppose $1 \cdot 1_A=1_A$ and $g^i \cdot 1_A =0$, for all $1 \leq i <n$.
	Then, $\cdot$ is a partial action of $T_n(q)$ on $A$ if and only if $(x \cdot 1_A)^n \in Z(A)$ and 
	\begin{eqnarray}\label{formula_corpo}
 g^{n-i}x^j \cdot a = (-1)^i q^{ \frac{i(i+1)}{2}} { j \choose i }_{q} (x \cdot 1_A)^{j-i}a(x \cdot 1_A)^i,
 	\end{eqnarray}
		for any $a \in A$ and $0 \leq i,j < n$.
\end{cor}

	\begin{proof}
		Since $g^i \cdot 1_A =0$, for all $1 \leq i <n$, then $g^{i+k}=0$ when $i+k\neq 0$ and $i+k\neq n$. Now, we analyze these two exceptional cases. By Theorem \ref{taft_principal}, we have
		\begin{itemize}
			\item If $i+k=0$, then $i=k=0$, \emph{i.e.},  $x^j=(x\cdot 1_A)^ja$, for all $0 \leq j <n$;
			
		\item If $i+k=n$, then $g^{i+k}\cdot a=a$ and so		
		$$ g^{n-k}x^j \cdot a = q^{-(n-k)j} (-1)^{k} q^{- \frac{k(k-1)}{2}} { j \choose k }_{q^{-1}} (x \cdot 1_A)^{j-k}a(x \cdot 1_A)^{k}.$$
		
		Then, by identity \eqref{id_q}, 
		
		$ g^{n-k}x^j \cdot a = (-1)^k q^{ \frac{k(k+1)}{2}} { j \choose k }_{q} (x \cdot 1_A)^{j-k}a(x \cdot 1_A)^k,$ for all $0\leq k<n$.
				\end{itemize}
	\end{proof}

\noindent{\textit{Remark.}} If $A= \k$, the above corollary  coincides with the partial action $\lambda_\alpha$, $\alpha\in\Bbbk$, of the result \cite[Theorem 3.1]{taft_corpo_revista}.

\begin{cor}\label{x_ponto_1_comuta}
	Let $\cdot:T_n(q) \o A \longrightarrow A$ be a linear map, and suppose $g \cdot 1_A \in \{0, 1_A\}$ and $x \cdot 1_A \in Z(A)$.
	Then, $\cdot$ is a partial action of $T_n(q)$ on $A$ if and only if  $\cdot$ is a global action, or $\cdot$ is a linear map such that $\cdot|_{\k C_n \otimes A}$
	is a partial group action such that $g \cdot 1_A = 0$ and, for all $i,j \in \I_{0,n-1}, a \in A$,
	\begin{itemize}
		\item[(i)] $ g^ix^j \cdot a = q^{-ij} (x \cdot 1_A)^j \sum_{k=0}^j (-1)^k q^{- \frac{k(k-1)}{2}} { j \choose k }_{q^{-1}}(g^{i+k} \cdot a);$
		\item[(ii)] $g^i \cdot (x \cdot 1_A) = q^{-i}(g^i \cdot 1_A)(x \cdot 1_A)$.
	\end{itemize}		
\end{cor}

	In particular, if $A$ is a commutative algebra, then Theorem \ref{taft_principal} and Corollary \ref{x_ponto_1_comuta} are equivalent.

\begin{cor}
		Let $\cdot:T_n(q) \o A \longrightarrow A$ be a linear map, and suppose $g \cdot 1_A \in \{0, 1_A\}$ and that there exists $k \in \I_{2,n-1}$ such that $g^k \cdot a = a$, for all $a \in A$.
		Then, $\cdot$ is a partial action of $T_n(q)$ on $A$ if and only if  $\cdot$ is a global action, or $\cdot$ is a linear map where $\cdot|_{\k C_n \otimes A}$
		is a partial group action such that $g \cdot 1_A = 0$ and  $g^ix^j \cdot a = \delta_{j,0} (g^i \cdot a)$,
		for all $i,j \in \I_{0,n-1}, a \in A$.
\end{cor}

\begin{proof}
By Theorem \ref{taft_principal} (iii), we get $x \cdot 1_A =0$.
\end{proof}

\begin{exa}\label{ex_taft}
Every partial action of Sweedler's $4$-dimensional Hopf algebra on $A$, namely $\cdot : \mathbb{H}_{4} \o A \longrightarrow A$, such that $g \cdot 1_A =0$, is given by
	$$1\cdot a=a, \ \ \ g\cdot a=0, \ \ \ x\cdot a= w a \ \ \ \mbox{and} \ \ \ gx\cdot a=a w,$$
	for all $a\in A$, where $w = x\cdot 1_A$ satisfies $w^2 \in Z(A)$.
	
In particular, if $w \in Z(A)$, then for all $a\in A$, $$1\cdot a=a, \ \ \ g\cdot a=0 \ \ \ \mbox{and} \ \ \ x\cdot a=gx\cdot a=w a.$$

When $A=\Bbbk$, every $w \in \k$ characterizes such a partial action by the linear map $\lambda_w: \mathbb{H}_{4}\longrightarrow\Bbbk$, defined as $\lambda_w (1)=1_{\Bbbk}, \lambda_w(g)=0, \lambda_w(x)=\lambda_w(gx)=w.$
	\end{exa}

\begin{exa}
Every partial action $\cdot : T_3(q) \o A \longrightarrow A$,
such that $g \cdot 1_A =0$, is given by
	$$1\cdot a=a, \ \ \ g\cdot a=g^2\cdot a=0, \ \ \ x\cdot a=wa, \ \ \ gx\cdot a=0,$$ $$g^2x\cdot a=-qaw, \   x^2\cdot a=w^2a, \ gx^2\cdot a=aw^2,  \ g^2x^2\cdot a=waw,$$
for all $a\in A$, where $w = x\cdot 1_A$ satisfies $w^3 \in Z(A)$.
	
In particular, if $w \in Z(A)$, then, for all $a\in A$, we have
$$1\cdot a=a, \ \ \ g\cdot a=g^2\cdot a=0, \ \ \ x\cdot a=wa \ \ \ gx\cdot a=0,$$
$$g^2x\cdot a=-q wa, \ x^2\cdot a= gx^2\cdot a=g^2x^2\cdot a=w^2a.$$
	
When $A=\Bbbk$, every $w \in \k$ characterizes such a partial action by the linear map
$\lambda_w: T_3(q)\longrightarrow\Bbbk$, where $\lambda_w(g^i)=\delta_{i,0}$, $\lambda_w(x)=w$, $\lambda_w(gx)=0$, $\lambda_w(g^2x)= - q w$ and $\lambda_w(g^ix^2)=w^2$, for all $i \in \I_{0,2}$.
\end{exa}

\begin{exa}
Assume that $\cdot : T_4(q) \o A \longrightarrow A$ is a partial action such that $g \cdot 1_A =0$. Then,
 $$1\cdot a=a, \  g\cdot a=g^3\cdot a=0, \ g^2\cdot a=  \alpha_{g^2}(a), \ x\cdot a=wa,$$ $$gx\cdot a=q(g^2\cdot a)w, \ g^2x\cdot a=-w(g^2\cdot a), \ g^3x\cdot a=-qaw,$$ $$x^2\cdot a=w^2a-q(g^2\cdot a)w^2, \ gx^2\cdot a=(1-q)w(g^2\cdot a)w,$$ $$g^2x^2\cdot a= w^2(g^2\cdot a)-qaw^2, \ g^3x^2= (1-q)waw,$$ $$x^3\cdot a=w^3a-w(g^2\cdot a)w^2, \ gx^3\cdot a=-w^2(g^2\cdot a)w+aw^3,$$ $$g^2x^3\cdot a=-w^3(g^2\cdot a)+waw^2, \ g^3x^3\cdot a=w^2aw-(g^2\cdot a)w^3,$$ 
where $\alpha_{g^2}: D_{g^2} \longrightarrow D_{g^2}$, being $D_{g^2} = (g^2 \cdot 1_A) A$, is the linear map inherited from the partial group action $\k C_4$ on $A$, and any element $w \in A$ such that $w^4\in Z(A)$ and $g^2\cdot w=-(g^2\cdot 1_A)w$.

In particular, some interesting cases are:
\begin{enumerate}
	\item If $g^2 \cdot 1_A = 0$, then 
$$1\cdot a=a, \ \ g\cdot a=g^2\cdot a=g^3\cdot a=0, \ \ x\cdot a=wa, \ \ gx\cdot a=g^2x\cdot a=0,$$
$$ g^3x\cdot a=-qaw, \ x^2\cdot a=w^2a, \ gx^2\cdot a=0, \ g^2x^2\cdot a=-qaw^2, g^3x^2=(1-q)waw,$$
$x^3\cdot a=w^3a, \ \ gx^3\cdot a=aw^3, \ \ g^2x^3\cdot a=waw^2, \ \ g^3x^3\cdot a=w^2aw.$

\item If $\alpha_{g^2} = Id_{D_{g^2}}$, then 	for all $0\leq i\leq n-1$ and $1\leq j\leq n-1$,
	$$1\cdot a=a, \ \ \ g\cdot a=g^3\cdot a=0, \ \ \ g^2\cdot a=a, \ \ \ g^ix^j\cdot a=0.$$

\item If $A=\Bbbk$, then the subcases (1) and (2) describe all the partial actions of $T_4(\omega)$ on $\k$, as in \cite[Example 3.8]{taft_corpo_revista}.
\end{enumerate}
\end{exa}	

For a partial action to define a partial representation, it must be symmetric \cite[Example 3.5]{Partial_representations}.
Our next goal is to verify that the partial actions presented in Theorem \ref{taft_principal} are symmetric.

\begin{thm}\label{simetria_taft}
	Every partial action $\cdot  : T_n(q) \o A \longrightarrow A$ such that $g \cdot 1_A \in \{0, 1_A\}$ is symmetric.
\end{thm}

\begin{proof}
	If $g \cdot 1_A = 1_A$, the partial action is, in fact, a global action, and so the result is clear.
	
	Now, assume $g \cdot 1_A = 0$. 
	It is sufficient to check the symmetric condition (PA.S):
	$g^ix^j \cdot ((g^tx^s \cdot a)b) = ((g^ix^j)_1 g^tx^s \cdot a)((g^ix^j)_2 \cdot b).$
	
	As in the proof of Proposition \ref{prop_volta}, we proceed by fixing vectors and checking the scalars. On one hand,	
	\begin{align*}
		g^i x^j & \cdot ((g^tx^s \cdot a)b) = q^{-ij} \sum_{k=0}^{j} q^{\frac{-k(k-1)}{2}} (-1)^k {j \choose k}_{q^{-1}} \!\!\!\! w^{j-k} \left( g^{i+k} \cdot (g^tx^s \cdot a)b \right) w^k\\
		\stackrel{\ref{lema_ind}}{=} \ & q^{-ij-s(t+i)} \sum_{k=0}^{j} \sum_{\ell=0}^{s} q^{-sk -\frac{k(k-1)}{2} - \frac{\ell(\ell-1)}{2}} (-1)^{k+\ell} {j \choose k}_{q^{-1}} {s \choose \ell}_{q^{-1}}\\
		\times \ & w^{j+s-(k+\ell)}(g^{i+t+k+\ell} \cdot a) w^{\ell}(g^{i+k}\cdot b)w^k.
	\end{align*}
	
		Considering $0 \leq k_0 \leq j$ and $0 \leq \ell_0 \leq s$ fixed, we obtain the vector
	\begin{eqnarray}\label{eqsimetria_vetor}
		w^{j+s-(k_0+\ell_0)}(g^{i+t+k_0+\ell_0} \cdot a) w^{\ell_0}(g^{i+k_0}\cdot b)w^{k_0}
	\end{eqnarray}
	with scalar	
	\begin{eqnarray}\label{eqsimetria}
		q^{-ij-s(t+i)} q^{-sk_0 -\frac{k_0(k_0-1)}{2} - \frac{\ell_0(\ell_0-1)}{2}} (-1)^{k_0+\ell_0} {j \choose k_0}_{q^{-1}} {s \choose \ell_0}_{q^{-1}}.
	\end{eqnarray}
	
	On the other hand, 
	\begin{align*}
		& ((g^ix^j)_1 g^t x^s \cdot a)((g^ix^j)_2  \cdot b) = \sum_{\ell=0}^{j} {j \choose \ell}_{q} (g^{i+\ell}x^{j-\ell}g^tx^s \cdot a)(g^ix^{\ell} \cdot b)\\
		= \ & \sum_{\ell=0}^{j} q^{\ell(j-\ell) + t(j-\ell)}{j \choose \ell}_{q^{-1}} (g^{i+t+\ell}x^{j+s-\ell} \cdot a)(g^ix^{\ell} \cdot b).
	\end{align*}
	Due to the terms $g^{i+t+\ell}x^{j+s-\ell} \cdot a$, we have to proceed by cases.
	
	\underline{Case 1: $j+s < n$.}
	\begin{align*}
		& \sum_{\ell=0}^{j} q^{\ell(j-\ell) + t(j-\ell)}{j \choose \ell}_{q^{-1}} (g^{i+t+\ell}x^{j+s-\ell} \cdot a)(g^ix^{\ell} \cdot b)\\
		= \ & q^{-ij-s(i+t)}\sum_{\ell=0}^{j} \sum_{r=0}^{j+s-\ell}\sum_{\theta=0}^{\ell} q^{-\ell s -\frac{r(r-1)}{2} - \frac{\theta(\theta-1)}{2}} (-1)^{r + \theta} {j \choose \ell}_{q^{-1}} \\
		\times \ &  {j+s-\ell \choose r}_{q^{-1}} {\ell \choose \theta}_{q^{-1}} w^{j+s-(\ell+r)} (g^{i+t+\ell+r} \cdot a) w^{r+\ell-\theta} (g^{i+\theta} \cdot b) w^{\theta}\\
	\stackrel{k = \ell + r}{=} & q^{-ij-s(i+t)}\sum_{\ell=0}^{j} \sum_{k=\ell}^{j+s}\sum_{\theta=0}^{\ell} q^{-\ell s - \frac{(k-\ell)(k-\ell-1)}{2} - \frac{\theta(\theta-1)}{2}} (-1)^{k+\theta -\ell}\\
		\times \ & {j \choose \ell}_{q^{-1}} {j+s-\ell \choose k-\ell}_{q^{-1}} {\ell \choose \theta}_{q^{-1}} w^{j+s-k} (g^{i+t+k} \cdot a) w^{k-\theta} (g^{i+\theta} \cdot b) w^{\theta}.
	\end{align*}	
	For $\theta = k_0$ and $k=k_0+\ell_0$, we obtain the following scalar
	\begin{align*}
		& \  q^{-ij-s(i+t)}\sum_{\ell=0}^{j} q^{-\ell s - \ell_0k_0+k_0\ell+\ell_0\ell  -k_0(k_0-1) - \frac{\ell_0(\ell_0-1)}{2} - \frac{\ell(\ell+1)}{2} } (-1)^{\ell_0 -\ell}\\
		& \times {j \choose \ell}_{q^{-1}} {j+s-\ell \choose k_0+\ell_0-\ell}_{q^{-1}} {\ell \choose k_0}_{q^{-1}}. 
	\end{align*}
	Therefore, this case holds by \eqref{eqsimetria} and identity \eqref{Lema2_10}.
	
		\underline{Case 2: $j+s \geq n$.}
	\begin{align*}
		& \sum_{\ell=0}^{j} q^{\ell(j-\ell) + t(j-\ell)}{j \choose \ell}_{q^{-1}} (g^{i+t+\ell}x^{j+s-\ell} \cdot a)(g^ix^{\ell} \cdot b)\\
		= \	& q^{-ij - s(i+t)} \sum_{\ell=j+s-n+1}^{j} \sum_{k=\ell}^{j+s} \sum_{\theta=0}^{\ell} q^{-\ell s - \frac{(k-\ell)(k-\ell-1)}{2} - \frac{\theta(\theta-1)}{2}} (-1)^{k - \ell + \theta}\\
		\times \ & {j \choose \ell}_{q^{-1}} {j+s-\ell \choose k-\ell}_{q^{-1}} {\ell \choose \theta}_{q^{-1}} w^{j+s-k} (g^{i+t+k} \cdot a ) w^{k-\theta}(g^{i+\theta} \cdot b)w^{\theta}.
	\end{align*}
	
	We will seek the vector highlighted in \eqref{eqsimetria_vetor} in the above expression and check its scalar.
	If $j+s-(\ell_0 + k_0) < n$, one proceeds as in Case 1.
	Thus, assume $j+s-(\ell_0 + k_0) \geq n$.
	For $\theta = k_0$ fixed,  $\ell \geq k_0$ and so $j+s-n+1 \geq k_0$. 
	Moreover, since $j+s-k < n$, for each $\ell \in \I_{j+s-n+1,j}$ there is a unique $k = \ell_0 + k_0+n$ such that the vector $w^{j+s-(k_0+\ell_0)-n} (g^{i+t+k_0+\ell_0} \cdot a ) w^{\ell_0+n}(g^{i+k_0} \cdot b)w^{k_0}$ appears in the above expression.
	This vector is the same in \eqref{eqsimetria_vetor} due the commutative of $w^n$, but with scalar
	\begin{align*}
		& q^{-ij - s(i+t)} \sum_{\ell=j+s-n+1}^{j} q^{\ell s - \frac{(k + \ell_0+n-\ell)(k+\ell_0+n-\ell-1)}{2} - \frac{k_0(k_0-1)}{2}} (-1)^{\ell_0 - \ell + n} \\
		& \times {j \choose \ell}_{q^{-1}} {j+s-\ell \choose k_0+\ell_0+n-\ell}_{q^{-1}} {\ell \choose k_0}_{q^{-1}} 
	\end{align*}
	
	Now, we investigate this scalar. First, by identity \eqref{igualdade_q_n} it is equal to
	\begin{align*}
		& \ q^{-ij - s(i+t)} \sum_{\ell=j+s-n+1}^{j}  q^{-\ell s - \ell_0k_0 + \ell(k_0+\ell_0) - k_0(k_0-1) - \frac{\ell_0(\ell_0-1)}{2} - \frac{\ell(\ell-1)}{2}} \\
		& \times (-1)^{\ell_0 - \ell+1} {j \choose \ell}_{q^{-1}} {j+s-\ell \choose k_0+\ell_0+n-\ell}_{q^{-1}} {\ell \choose k_0}_{q^{-1}},
	\end{align*}
	and, comparing it with the scalar in \eqref{eqsimetria}, we only have to show that:
	\begin{align*}
		& \sum_{\ell=j+s-n+1}^{j}  q^{-\ell s +sk_0 - \ell_0k_0 + \ell(k_0+\ell_0) -\frac{k_0(k_0-1)}{2} - \frac{\ell(\ell-1)}{2}} (-1)^{k_0- \ell+1}\\
		& \times {j \choose \ell}_{q^{-1}} {j+s-\ell \choose k_0+\ell_0+n-\ell}_{q^{-1}} {\ell \choose k_0}_{q^{-1}} = {j \choose k_0}_{q^{-1}} {s \choose \ell_0}_{q^{-1}}.
	\end{align*}
	By identity \eqref{Lema2_10}, the expression above is equivalent to:
	\begin{align*}
		& \sum_{\ell=0}^{j+s-n}  q^{-\ell s +sk_0 - \ell_0k_0 + \ell(k_0+\ell_0) -\frac{k_0(k_0-1)}{2} - \frac{\ell(\ell-1)}{2}} (-1)^{k_0- \ell}\\
		& \times  {j \choose \ell}_{q^{-1}} {j+s-\ell \choose k_0+\ell_0+n-\ell}_{q^{-1}} {\ell \choose k_0}_{q^{-1}} = {j \choose k_0}_{q^{-1}} {s \choose \ell_0}_{q^{-1}}.
	\end{align*}
	
	Note that, on the left side of the above equality, due to the term ${ \ell \choose k_0}_{q^{-1}}$, the lower index of the summation symbol can begin in $k_0$ instead $0$. Then, we replace $\ell$ by $L+k_0$ and use \eqref{Lema2_2}, to get the equivalent expression 
	\begin{align*}
		&  \sum_{L=0}^{j+s-n-k_0}  q^{L(\ell_0-s) -\frac{L(L-1)}{2}} (-1)^{L} {j \choose k_0}_{q^{-1}} {j-k_0 \choose L}_{q^{-1}} {j+s-L-k_0 \choose \ell_0+n-L}_{q^{-1}} \\
		= \ & {j \choose k_0}_{q^{-1}} {s \choose \ell_0}_{q^{-1}}.
	\end{align*}
	Finally, with a change of variable on the upper index of the summation, the above expression is true by identity \eqref{identidade_R_L_item}.
\end{proof}

\subsection{Partial Coactions}

\medbreak 

In this section, we investigate the theory of the partial coactions of Taft algebras on algebras.

First, recall that $T_n(q)$ is a self-dual Hopf algebra. 

\begin{lem}\label{inversa}
	Consider the Taft algebra $T_n(q)$ and the linear maps:
	\begin{itemize}
		\item[(1)] $\psi : T_n(q) \longrightarrow (T_n(q))^*$ given by
		$\psi(g^ix^j)=G^iX^j,$ where 
		\begin{align*}
		G^iX^j  = \sum_{k=0}^{n-1} (j)_q! q^{-i(k+j)-jk-\frac{j(j-1)}{2}} (g^kx^j)^*;
		\end{align*}
		
		\item[(2)] $\varphi : (T_n(q))^* \longrightarrow T_n(q)$ given by
		\begin{align*}
		\varphi ((g^ix^j)^*)=\frac{1}{n}((j)_q!)^{-1} q^{ij+\frac{j(j-1)}{2}} \sum_{k=0}^{n-1} q^{k(i+j)} g^kx^j. 
		\end{align*}
	\end{itemize}
	Then, both maps defined in (1) and (2) are isomorphisms of Hopf algebras.
	Furthermore, $\varphi = \psi^{-1}$.
\end{lem}

For the next result, consider $\pi$ the canonic projection from $T_n(q)$ on $\Bbbk C_n$, \emph{i.e.}, $\pi : T_n(q) \longrightarrow \Bbbk C_n$ given by $\pi(g^ix^j)=\delta_{j,0}g^i$, for all $i,j \in \I_{0,n-1}$.
\begin{thm}\label{coacao_taft}
	Let $\rho:A \longrightarrow A \o T_n(q)$ be a linear map, denoted by $\rho(a)=a^0\otimes a^1$, such that $A_1\in\{0,1\}$, where $A_i=\sum_{r=0}^{n-1}q^{-ir}(g^r)^{\ast}(1^1)1^0$, $0 \leq i <n$. Then, $\rho$  is a partial coaction if and only if $\rho$ is a global coaction, or $(I\otimes\pi)\rho$ is a partial coaction of  $\Bbbk C_n$ on $A$ such that $A_1=0$ and
	\begin{itemize}
		\item[(i)] $\displaystyle \rho(a)= \sum\limits_{i,j = 0}^{n-1}\sum_{k=0}^j q^{-ij} (-1)^k q^{- \frac{k(k-1)}{2}} { j \choose k }_{q^{-1}} w^{j-k}\\\times\left(\sum_{t=0}^{n-1}q^{-(i+k)t}(g^t)^{\ast}(a^1)a^0\right) w^k 
		\otimes\frac{1}{n}((j)_q!)^{-1} q^{ij+\frac{j(j-1)}{2}} \sum_{\ell=0}^{n-1} q^{\ell(i+j)} g^\ell x^j$;
		\item[(ii)] $w^n \in Z(A);$
		\item[(iii)] $\sum_{s=0}^{n-1}q^{-is}(g^s)^{\ast}(w^1)w^0=q^{-i}A_iw,$
		\end{itemize}
	for all  $a \in A$.
\end{thm}

\begin{proof}
	By Proposition \ref{dual} and Lemma \ref{inversa}, $A$ is a partial $T_n(q)$-comodule algebra  via $\cdot$
	if and only if, $A$ is a partial $T_n(q)$-module algebra via $\rho$. 
	
	Notice that the maps $\cdot$ and $\rho$, when considered just as linear, are dual.
	
	Therefore, considering $A_i=\sum_{r=0}^{n-1}q^{-ir}(g^r)^{\ast}(1^1)1^0$,
	\begin{eqnarray*}
		A_1=0&\Leftrightarrow & g \cdot 1_A=0  \ \ \mbox{and} \ \ 
		A_1=1\Leftrightarrow  g \cdot 1_A=1.
	\end{eqnarray*}
	
	By Proposition \ref{dual} and Theorem \ref{taft_principal}, 
	\begin{eqnarray*}
		\rho(a)&=& \sum\limits_{i,j} (g^ix^j)\cdot a \otimes \varphi((g^ix^j)^{\ast})\\
		&=&\sum\limits_{i,j}\sum_{k=0}^j q^{-ij} (-1)^k q^{- \frac{k(k-1)}{2}} { j \choose k }_{q^{-1}} (x \cdot 1_A)^{j-k}(g^{i+k} \cdot a)(x \cdot 1_A)^k\\
		&\otimes &  \frac{1}{n}((j)_q!)^{-1} q^{ij+\frac{j(j-1)}{2}} \sum_{\ell=0}^{n-1} q^{\ell(i+j)} g^\ell x^j, 
	\end{eqnarray*}
	for all $a\in A$ and $\varphi$ given by Lemma \ref{inversa}. 
	
	Moreover, $\cdot|_{\k C_n \otimes A}$
	is a partial group action such that $g \cdot 1_A = 0$ if and only if $(I\otimes\pi)\rho$ is a partial coaction of  $\Bbbk C_n$ on $A$ such that $A_1=0$.
	
	Finally, considering  $w=x\cdot 1_A$,
\begin{eqnarray*}
			\sum_{s=0}^{n-1}q^{-is}(g^s)^{\ast}(w^1)w^0=q^{-i}A_iw \Leftrightarrow  g^i \cdot (x \cdot 1_A) = q^{-i}(g^i \cdot 1_A)(x \cdot 1_A).
	\end{eqnarray*}
\end{proof}

\noindent{\textit{Remark.}} Since the partial actions given by Theorem \ref{simetria_taft} are symmetric, it follows by Proposition \ref{dual} that the partial coactions  of the above theorem are symmetric.

\begin{cor}\label{idempotente_trivial_coac}
	Suppose that the algebra $A$ has only the trivial idempotents. Then, Theorem \ref{coacao_taft} characterizes all partial coactions of $T_n(q)$ on A.
\end{cor}

\begin{cor}\label{formula_coacao_corpo}
	Let $\rho:A \longrightarrow A \o T_n(q)$ be a linear map, denoted by $\rho(a)=a^0\otimes a^1$, and suppose $A_0=1$ and $A_i =0$, for all $1 \leq i <n$, where $A_i=\sum_{r=0}^{n-1}q^{-ir}(g^r)^{\ast}(1^1)1^0$. Then, $\rho$  is a partial coaction  of $T_n(q)$ on $A$ if and only if $w^n \in Z(A)$ and 
	\begin{equation*}
	\displaystyle \rho(a)= \sum_{i,j =0}^{n-1} (-1)^i q^{ \frac{i(i+1)}{2}} { j \choose i }_{q} w^{j-i}a w^i  			\otimes \frac{1}{n}((j)_q!)^{-1} q^{-ij+\frac{j(j-1)}{2}} \sum_{k=0}^{n-1} q^{k(j-i)} g^kx^j,
	\end{equation*}
	for all $a \in A$.
\end{cor}

\noindent{\textit{Remark.}} If $A= \k$, this partial coaction coincides with the coaction $z_{\alpha}$, $\alpha\in\Bbbk$, of \cite[Theorem 3.13]{taft_corpo_revista}.

\begin{cor}\label{w_comuta}
	Let $\rho:A \longrightarrow A \o T_n(q)$ be a linear map, denoted by $\rho(a)=a^0\otimes a^1$, such that $A_1\in\{0,1\}$, where $A_i=\sum_{r=0}^{n-1}q^{-ir}(g^r)^{\ast}(1^1)1^0$, $0 \leq i <n$, and suppose $w =\sum_{k=0}^{n-1} q^{-k}(g^kx)^{\ast}(1^1)1^0 \in Z(A)$.
	Then, $\rho$ is a partial coaction of $T_n(q)$ on $A$ if and only if  $\rho$ is a global coaction, or $(I\otimes\pi)\rho$ is a partial coaction of  $\Bbbk C_n$ on $A$ such that $A_1=0$ and
	\begin{itemize}
		\item[(i)] $\displaystyle \rho(a)= \sum\limits_{i,j = 0}^{n-1} w^j \sum_{k=0}^j (-1)^k q^{- \frac{k(k-1)}{2}} { j \choose k }_{q^{-1}} \left(\sum_{t=0}^{n-1}q^{-(i+k)t}(g^t)^{\ast}(a^1)a^0\right)\\
		\otimes \frac{1}{n}((j)_q!)^{-1} q^{\frac{j(j-1)}{2}} \sum_{\ell=0}^{n-1} q^{\ell(i+j)} g^\ell x^j$;
		\item[(ii)] $ \sum_{s=0}^{n-1}q^{-is}(g^s)^{\ast}(w^1)w^0=q^{-i}A_iw,$
	\end{itemize}		
	for all $a \in A$.
\end{cor}

In particular, if $A$ is a commutative algebra, then Theorem \ref{coacao_taft} and Corollary \ref{w_comuta} are equivalent.

\begin{cor}\label{coacao_identidade}
	Let $\rho:A \longrightarrow A \o T_n(q)$ be a linear map, denoted by $\rho(a)=a^0\otimes a^1$, such that $A_1\in\{0,1\}$, where $A_i=\sum_{r=0}^{n-1}q^{-ir}(g^r)^{\ast}(1^1)1^0$, $i \in \I_{0,n-1}$, and suppose  there exists $k \in I_{2,n-1}$, such that $\sum_{r=0}^{n-1}q^{-kr}(g^r)^{\ast}(a^1)a^0= a$, for all $a \in A$.
	Then, $\rho$ is a partial coaction of $T_n(q)$ on $A$ if and only if  $\rho$ is a global coaction, or $\rho$ is a linear map such that $(I \otimes \pi) \rho$
	is a partial coaction of  $\Bbbk C_n$ on $A$ such that $A_1=0$ and $$\rho(a)= \sum\limits_{i= 0}^{n-1} \left(\sum_{t=0}^{n-1}q^{-it}(g^t)^{\ast}(a^1)a^0\right)
	\otimes \frac{1}{n} \sum_{\ell=0}^{n-1} q^{\ell i} g^\ell,$$
	for all $a \in A$.
\end{cor}

\begin{proof}
	By hypothesis it follows that $A_k=1$ and so, by Theorem \ref{coacao_taft} (iii), $w=0$.
\end{proof}

\begin{exa}		\label{exa_coactiontaft}
	Every partial coaction of Sweedler's $4$-dimensional Hopf algebra on $A$, namely $\rho : A \longrightarrow A \o \mathbb{H}_{4}$, such that $A_1 =0$, is given by
	$$\rho(a)=a\otimes \left(\frac{1+g}{2}\right)+ w a\otimes \left(\frac{x-gx}{2}\right)-a w \otimes \left(\frac{x+gx}{2}\right),$$
	for all $a\in A$, where $w^2 \in Z(A)$.
	
	In particular, if $w \in Z(A)$, then for all $a\in A$, $$\rho(a)=a\otimes \left(\frac{1+g}{2}\right)- a w\otimes gx.$$
		
	For instance, if $A=\Bbbk[z]$ is the polynomial algebra in the variable $z$, this example generalizes \cite[Example 3.14]{corepresentations}.
	
	When $A=\Bbbk$, every $w \in \k$ characterizes such a partial coaction $\rho(1)=1\otimes z_w$, where $z_w= \frac{1 + g}{2} - w gx$.
\end{exa}

\begin{exa}
	Every partial coaction $\rho : A \longrightarrow A \o T_3(q)$, such that $A_1 =0$, is given by
		\begin{eqnarray*}
		\rho(a)&=&a\otimes \left(\frac{1+g+g^2}{3}\right)+w a \otimes \left(\frac{x+qgx+q^2g^2x}{3}\right)\\
		&-&aw\otimes \left(\frac{x+gx+g^2x}{3}\right)-w^2a\otimes \frac{q^2}{3}\left(x^2+q^2gx^2+qg^2x^2\right)\\
		&-& aw^2\otimes \frac{q}{3}\left(x^2+gx^2+g^2x^2\right)-waw\otimes \frac{1}{3}\left(x^2+qgx^2+q^2g^2x^2\right),
	\end{eqnarray*}
	for all $a\in A$, where $w^3 \in Z(A)$.
	
	In particular, if $w \in Z(A)$, then, for all $a\in A$,
	\begin{eqnarray*}
		\rho(a)&=&a\otimes \left(\frac{1+g+g^2}{3}\right)+ w a\otimes \frac{1}{3}\left( (q-1)gx + (q^2-1) g^2x\right)\\
		&-&wa\otimes qgx^2.
	\end{eqnarray*}
	
	When $A=\Bbbk$, every $w \in \k$ characterizes such a partial coaction $\rho(1)=1\otimes z_w$, where $z_w = \frac{1 + g + g^2 }{3} + \frac{1}{3}\left( (q-1)w gx + (q^2-1) w g^2x - 3q w^2gx^2 \right).$
\end{exa}

\begin{exa}
	The partial coaction of $T_4(\omega)$ on $A$ given by
	$$\rho(a)=a\otimes \left(\frac{1+g^2}{2}\right),$$
	for all $a \in A$, is an example that satisfies the conditions of Corollary \ref{coacao_identidade}.
\end{exa}

\section{Partial (co)actions of the Nichols Hopf algebra}\label{sec:Nichols}

First, we shall emphasize that Nichols Hopf algebras were introduced by Taft in \cite{Taft}, but were named after the work of Nichols \cite{nichols}.
Such Hopf algebras are the prototype of the theory of \emph{Nichols algebras}.
We will present this family of Hopf algebras as it appears in \cite[Section 2.2]{etingof}.

\medbreak

Let $n \geq 2$ be an integer and suppose $char(\k) \neq 2$.
The \emph{Nichols Hopf algebra of order $n$}, or shortly \emph{Nichols Hopf algebra}, here denoted by $\mathbb{H}_{2^n}$, has the structure as follows:
as algebra it is generated over $\k$ by the $n$ letters $g,x_1, \cdots, x_{n-1}$ with relations $g^2=1$, $x_i^2=0$, $x_i g = -g x_i$ and $x_ix_j = -x_jx_i$, for all $i, j \in \I_{n-1}$. 
Thus, the set
$\mathcal{B} = \{g^{j_0}x_1^{j_1}x_2^{j_2}...x_{n-1}^{ j_{n-1} }  : j_i \in \I_{0,1}, i \in \I_{0, n-1} \}$
is the canonical basis for $\mathbb{H}_{2^n}$ and consequently $dim_\k(\mathbb{H}_{2^n}) = 2^{n}$.
To complete the Hopf algebra structure of $\mathbb{H}_{2^n}$, we set
$\Delta (g) = g \o g$, $\varepsilon (g) = 1$, $S(g) =g^{-1}=g,$ and $\Delta (x_i) = x_i \o 1 + g \o x_i$, $\varepsilon (x_i) = 0$ and $S(x_i) = -gx_i,$ for all $i \in \I_{n-1}$. 
Note that  $G(\mathbb{H}_{2^n})= \{1, g \} = C_2$, for any $n$.
In particular, when $n=2$, the Nichols Hopf algebra $\mathbb{H}_{2^2}$ is exactly the Sweedler's $4$-dimensional Hopf algebra $\mathbb{H}_{4}$.

 Similarly to global actions of Taft's algebra \cite[Proposition 4]{Acoes_taft_Centrone}, it is straightforward to determine global actions of Nichols Hopf algebra completely through an automorphism and skew-derivations.
 
 \begin{prop}
 	Let A be an algebra, then an action of the Nichols Hopf algebra $H_{2^n}$ on $A$ is completely determinate by a choice of:
 	\begin{itemize}
 		\item [(i)] An automorphism $\alpha$ of $A$ of order $2$, \textit{i.e.}, $\alpha^2=\alpha$;
 		\item [(ii)] A family of $\alpha$-derivations $\{\delta_i\}_{i \in \mathbb{I}_{n-1}}$ of $A$ such that ${\delta_i}^2=0$ and $\alpha  \delta_i = - \delta_i  \alpha$, for all $i \in \mathbb{I}_{n-1}$.
 	\end{itemize}
 	Equivalently, the structure of $H_{2^n}$-module algebra on A is uniquely determined by a choice of:
 	\begin{itemize}
 		\item[(i)] A $\mathbb{Z}_2$-grading $A=A_{0}\oplus A_{1}$, defined by an automorphism $\alpha$;
 		\item [(ii)] A family of $\alpha$-derivations $\{\delta_i\}_{i \in \mathbb{I}_{n-1}}$ such that ${\delta_i}^2=0$, $\delta_i(A_0) \subset A_1$ and $\delta_i(A_1) \subset A_0$, for all $i \in \mathbb{I}_{n-1}$.
 	\end{itemize}
 \end{prop}

\subsection{Partial Actions}
In this section, we compute a family of examples of partial actions of $\mathbb{H}_{2^n}$ on $A$. 

First, we shall analyze when a partial action is global.

\begin{prop}\label{prop_nichols_global}
	Let $\cdot : \mathbb{H}_{2^n} \o A \longrightarrow A$ be a partial action of $\mathbb{H}_{2^n}$ on $A$.
	If $g \cdot 1_A = 1_A$, then $\cdot$ is global.
\end{prop}
\begin{proof}
	Analogously to the proof of Proposition \ref{parte_global}.
\end{proof}

\begin{prop}\label{nichols}
	Let $\cdot : \mathbb{H}_{2^n} \o A \longrightarrow A$ be a partial action of $\mathbb{H}_{2^n}$ on $A$.
	If $g \cdot 1_A = 0$ and $x_i\cdot 1_A\in Z(A)$, $i \in \I_{n-1}$, then,
	for all $a \in A$,
	$$ gx_i\cdot a=x_i\cdot a=(x_i\cdot 1_A)a,$$
	$$gx_{i_1}x_{i_2}...x_{i_s}\cdot a=x_{i_1}x_{i_2}...x_{i_s}\cdot a=0, \ \ \ s \geq 2, i_\ell \in \I_{n-1}, \ell \in \I_{s}.$$
\end{prop}
\begin{proof}
	Let $\cdot : \mathbb{H}_{2^n} \o A \longrightarrow A$ be a partial action of $\mathbb{H}_{2^n}$ on $A$ with
	$x_i\cdot 1_A\in Z(A)$, $i \in \I_{n-1}$, and $g \cdot 1_A = 0$.
	
	First, note that $g\cdot a=0$, for $a\in A$. Then, by (PA.2), we have that $x_i\cdot a=(x_i\cdot 1_A)a$ and $gx_i\cdot a=a(gx_i\cdot 1_A)$. Moreover, considering $h=gx_i$ and $k=g$ in (PA.3), we obtain $gx_i\cdot 1_A=x_i\cdot 1_A$, and using that $x_i\cdot 1_A\in Z(A)$ we conclude $gx_i\cdot a=x_i\cdot a=(x_i\cdot 1_A)a$, for all $i \in \I_{n-1}$ and $a\in A$.
	
	Now, for $h=gx_i$ and $k=x_j$ we get $x_ix_j\cdot a=0$, and finally taking $h=gx_i$ and $k=gx_j$ we deduce that $gx_ix_j\cdot a=0$, for all $i,j \in \I_{n-1}$, $i \neq j$.
	
	By induction, one can prove that $x_{i_1}x_{i_2}...x_{i_s}\cdot a=gx_{i_1}x_{i_2}...x_{i_s}\cdot a=0$,  for $s \geq 2$, $i_\ell \in \I_{n-1}$, $\ell \in \I_{s}$.
	Indeed, we already done for $s=2$.
	Assume $s \geq 2$ and consider $h=gx_{i_1}$, $k=x_{i_2}...x_{i_s}$ (resp. $h=gx_{i_1}$, $k=gx_{i_2}...x_{i_s}$) in (PA.3).
	Then, using the induction hypothesis, it follows that $gx_{i_1}x_{i_2}...x_{i_s}\cdot a=0$ (resp. $x_{i_1}x_{i_2}...x_{i_s}\cdot a=0$). 	
\end{proof}

\begin{prop}\label{lda_alpha_eh_acao_nichols}
	Let $\cdot : \mathbb{H}_{2^n} \o A \longrightarrow A$ be a linear map  such that $1 \cdot 1_A = 1_A$, $g \cdot 1_A = 0$, $x_i\cdot 1_A\in Z(A)$, $i \in \I_{n-1}$,  and for all $a \in A$,
	$$ gx_i\cdot a=x_i\cdot a=(x_i\cdot 1_A)a,$$
	$$gx_{i_1}x_{i_2}...x_{i_s}\cdot a=x_{i_1}x_{i_2}...x_{i_s}\cdot a=0, \ \ \ s \geq 2, i_\ell \in \I_{n-1}, \ell \in \I_{s}.$$
	Then,  $\cdot $ is a partial action of $\mathbb{H}_{2^n}$ on $A$.
\end{prop}
\begin{proof}
	We need to check if condition (PA.2$^\prime$) holds.
	We proceed by fixing $h \in \mathcal{B}$ and verifying (PA.2$^\prime$) for any $k \in \mathcal{B}$, where $\mathcal{B}$ is the canonical basis of $\mathbb{H}_{2^n}$.
	
	First, if $h \in \{1, g, x_1, \cdots, x_{n-1} \}$, then (PA.2$^\prime$) holds trivially for any $k \in \mathcal{B}$.

	If $h=gx_i$, $i \in I_{n-1}$, then (PA.2$^\prime$) means 
	$gx_i\cdot(a(k\cdot b)) =(gx_i\cdot a)(gk\cdot b) + a(gx_ik \cdot b),$
	for any $k \in \mathcal{B}$.
	Using  that $x_i\cdot 1_A\in Z(A)$, $i \in \I_{n-1}$, and considering $k$ as $g^\ell, g^\ell x_j$ and $g^\ell x_{j_1}x_{j_2}...x_{j_s}$, where $\ell \in \I_{0,1},$ $j \in \I_{n-1}$ and $s \in \I_{2,n-1}, j_t \in \I_{n-1}, t \in \I_s$, it is a routine computation to get the equality.
	
	Now, if $h=g^\ell x_{i_1}x_{i_2}$, $\ell \in \I_{0, 1}$, $i_1, i_2 \in \I_{n-1}$, $i_1 < i_2$, then
	\begin{align*}
		\Delta(g^\ell x_{i_1}x_{i_2}) = & \ g^\ell x_{i_1}x_{i_2} \o g^\ell - g^{\ell+1}x_{i_1} \o g^\ell x_{i_2} + g^{\ell +1} x_{i_2} \o g^{\ell}x_{i_1} \\
		+ &\  g^\ell \o g^\ell x_{i_1}x_{i_2},
	\end{align*}
	and so (PA.2$^\prime$) becomes
	\begin{align*}
		g^\ell x_{i_1}x_{i_2}\cdot (a(k\cdot b)) = & \ (g^\ell x_{i_1}x_{i_2}\cdot a)(g^\ell k\cdot b) - (g^{\ell+1}x_{i_1}\cdot a) (g^\ell x_{i_2} k\cdot b) \\
		+ & \ (g^{\ell+1} x_{i_2}\cdot a) (g^{\ell}x_{i_1} k\cdot b) +  (g^\ell\cdot a) ( g^\ell x_{i_1}x_{i_2} k\cdot b),
	\end{align*}
	for any $k \in \mathcal{B}$.
	
	Since $(gx_s\cdot a)=(x_s\cdot a)=(x_s\cdot 1_A)a$ and $(g^\ell x_{i_1}x_{i_2}\cdot a)=(g^\ell x_{i_1}x_{i_2}k\cdot a)=0$, for any $s \in \I_{n-1}$, $k \in \mathcal{B}$, (PA.2$^\prime$) is reduced to
	$$(x_{i_1}\cdot 1_A)a(g^\ell x_{i_2}k\cdot b) = (x_{i_2}\cdot 1_A)a (g^\ell x_{i_1}k\cdot b),$$
	for any $ k \in \mathcal{B}$.
	One can verify this equality using $x_i\cdot 1_A\in Z(A)$ for $i \in \I_{n-1}$ and considering $k$ as $g^t$ and $g^t x_{j_1}...x_{j_s}$, where $t \in \I_{0,1}$ and $s \in \I_{n-1}$, $j_k \in \I_{n-1}$, $k \in \I_{s}$, $j_1 < \cdots < j_s$.
	
	Finally, it remains to check (PA.2$^\prime$) for $h = g^\ell x_{i_1}x_{i_2}...x_{i_s}$, where $\ell \in \I_{0,1}$ and $s \geq 3$, and any $k \in \mathcal{B}$. 
	Write $h =g^\ell x_{i_1}x_{i_2}x_{i_3}w$, for some $w \in \mathbb{H}_{2^n}$.
	Then,
	\begin{align*}
		\Delta(g^\ell  x_{i_1}x_{i_2}x_{i_3}w) & =  \Delta(g^\ell  x_{i_1}x_{i_2})\Delta(x_{i_3}) \Delta(w) \\
		& =  g^\ell x_{i_1}x_{i_2}x_{i_3}w_1 \o g^\ell w_2 - g^{\ell +1}x_{i_1}x_{i_3}w_1 \o g^\ell x_{i_2}w_2 \\
		& + g^{\ell +1} x_{i_2}x_{i_3}w_1 \o g^{\ell }x_{i_1}w_2 + g^\ell  x_{i_3}w_1 \o g^\ell x_{i_1}x_{i_2}w_2 \\
		& + g^{\ell +1}x_{i_1}x_{i_2}w_1 \o g^\ell x_{i_3} w_2  + g^{\ell +2}x_{i_1}w_1 \o g^\ell x_{i_2}x_{i_3} w_2 \\
		& - g^{\ell +2} x_{i_2}w_1 \o g^{\ell }x_{i_1}x_{i_3} w_2 + g^{\ell +1} w_1 \o g^\ell x_{i_1}x_{i_2}x_{i_3} w_2,
	\end{align*}
	and for any $k \in \mathcal{B}$, condition (PA.2$^\prime$) means
	\begin{align*}
		g^\ell  x_{i_1}x_{i_2}x_{i_3}w\cdot (a(k\cdot b)) & =( g^\ell x_{i_1}x_{i_2}x_{i_3}w_1\cdot a)  (g^\ell w_2 k\cdot b) \\
		& -  ( g^{\ell +1}x_{i_1}x_{i_3}w_1\cdot a)  (g^\ell x_{i_2}w_2 k\cdot b)\\
		& +  ( g^{\ell +1} x_{i_2}x_{i_3}w_1\cdot a)  ( g^{\ell }x_{i_1}w_2 k\cdot b) \\
		& + ( g^\ell  x_{i_3}w_1\cdot a)( g^\ell x_{i_1}x_{i_2}w_2 k\cdot b) \\
		& + ( g^{\ell +1}x_{i_1}x_{i_2}w_1 \cdot a)  ( g^\ell x_{i_3} w_2 k\cdot b) \\
		& + ( g^{\ell +2}x_{i_1}w_1\cdot a) ( g^\ell x_{i_2}x_{i_3} w_2 k\cdot b) \\
		& - ( g^{\ell +2} x_{i_2}w_1 \cdot a) ( g^{\ell }x_{i_1}x_{i_3} w_2 k\cdot b) \\
		& + ( g^{\ell +1} w_1 \cdot a) ( g^\ell x_{i_1}x_{i_2}x_{i_3} w_2 k\cdot b).
	\end{align*}
	
	Note that the above equality holds since the linear map $\cdot$ evaluated in a product that contains at least two skew-primitive elements as factors is equal to zero. 	Therefore,  the map $\cdot$ as given is a partial action of $\mathbb{H}_{2^n}$ on $\k$.
\end{proof}

Considering the results presented in this section, we have characterized all partial actions $\cdot : \mathbb{H}_{2^n} \o A \longrightarrow A$ when $g\cdot 1_A = 0$ and $x_i\cdot 1_A\in Z(A)$ for all $i \in \I_{n-1}$.

\begin{thm}\label{simetria_nichols}
	Every partial action $\cdot$ of the Nichols Hopf algebra $\mathbb{H}_{2^n}$ on $A$ such that $g\cdot 1_A = 1_A$, or $g\cdot 1_A = 0$ and $x_i\cdot 1_A\in Z(A)$ for all $i \in \I_{n-1}$ is symmetric.
\end{thm}
\begin{proof}
	Let $\cdot$ be a partial action of the Nichols Hopf algebra $\mathbb{H}_{2^n}$ on $A$. If $g\cdot 1_A=1_A$, then, by Proposition \ref{prop_nichols_global}, $\cdot$ is a global action and so symmetric.
	
	If $g\cdot 1_A=0$ and  $x_i\cdot 1_A\in Z(A)$ for all $i \in \I_{n-1}$, the symmetric condition is verified similarly as the proof of Proposition \ref{lda_alpha_eh_acao_nichols}.
\end{proof}

\begin{exa} Every partial action of Nichols Hopf algebra of order $2$ on $A$, namely $\cdot : \mathbb{H}_{2^2} \o A \longrightarrow A$, such that $g \cdot 1_A =0$ and $x\cdot 1_A \in Z(A)$ is given by $$1\cdot a=a, \ \ \ g\cdot a=0 \ \ \ \mbox{and} \ \ \ x\cdot a=gx\cdot a=(x\cdot 1_A) a,$$
	for all $a\in A$.
In fact, since the Nichols Hopf algebra of order $2$ is the Sweedler's Hopf algebra, it is expected by Example \ref{ex_taft}.
\end{exa}

\begin{exa}
	Every partial action $\cdot : \mathbb{H}_{2^3} \o A \longrightarrow A$,
	such that $g \cdot 1_A =0$ and  $x_1\cdot 1_A, x_2\cdot 1_A\in Z(A)$, is given by
	$$1\cdot a=a, \ \ \ g\cdot a=0, \ \ \ x_1\cdot a=gx_1\cdot a=(x_1\cdot 1_A)a,$$ $$  x_2\cdot a=gx_2\cdot a=(x_2\cdot 1_A)a, \ \ \ x_1x_2\cdot a=gx_1x_2\cdot a=0,$$
	for all $a\in A$.
	
	When $A=\Bbbk$, every $\alpha=(\alpha_1, \alpha_2) \in \k^2$ characterizes such a partial action by the linear map  $\lambda_\alpha: \mathbb{H}_{2^3} \longrightarrow \Bbbk$, where $\lambda_\alpha(1)=1_\k$, $\lambda_\alpha(g)=0$, $\lambda_\alpha(x_1x_2)=0$, $\lambda_\alpha(gx_1x_2)=0$ and $\lambda_\alpha(x_i)=\lambda_\alpha(gx_i)=\alpha_i$, for all $i \in \I_{2}$.
\end{exa}

\subsection{Partial coactions}
In this section, we compute some  partial coactions of Nichols Hopf algebra $\mathbb{H}_{2^n}$ on an algebra $A$. 
The setting here is the same as for the Taft algebra, since both are self-dual Hopf algebras.

\begin{lem}\label{iso_nichols}
	The linear map  $\psi : \mathbb{H}_{2^n}  \longrightarrow (\mathbb{H}_{2^n} )^*$
	given by  $\psi(g)=1^* - g^*$ and $\psi(x_i)=x_i^* - (gx_i)^*$, for all $i \in \I_{n-1}$, is an isomorphism of Hopf algebras.
\end{lem}

From the isomorphism $\psi$ in the above lemma, it follows that $$\psi\left(\frac{1+g}{2}\right) = 1^*, \psi\left(\frac{x_i-gx_i}{2}\right) = x_i^* \textrm{ and } \psi\left(\frac{-(x_i+gx_i)}{2}\right) = (gx_i)^*,$$
for all $i \in \I_{n-1}$.
Thus, we are able to prove the following result.
\begin{thm}\label{nichols_coac}
	Let $\rho:A \longrightarrow A \o \mathbb{H}_{2^n}$ be a linear map, denoted by $\rho(a)=a^0\otimes a^1$, such that $A_1=(1^{\ast}-g^{\ast})(1^1)1^{0}=0$ and $w_i \in Z(A)$, where $w_i=(x_i^* - (gx_i)^*)(1^1)1^0$, for all $i \in \I_{n-1}$. Then, $\rho$  is a partial coaction if and only if
	\begin{equation*}
		\displaystyle \rho(a)= a\otimes \left(\frac{1+g}{2}\right)-\sum\limits_{i=1}^{n-1}w_ia\otimes gx_i,
	\end{equation*}
	for any $a \in A$.
	Moreover, all these partial coactions are symmetric.
\end{thm}

\begin{proof}
	Let $\mathcal{B} = \{g^{j_0}x_1^{j_1}x_2^{j_2}...x_{n-1}^{ j_{n-1} }  : j_i \in \I_{0,1}, i \in \I_{0, n-1} \}$ be the canonical basis of $\mathbb{H}_{2^n}$ .	By Proposition \ref{dual} and Proposition \ref{nichols}, $A_1=0\Leftrightarrow g \cdot 1_A=0$ and
	\begin{eqnarray*}
		\rho(a)&=& (g^{j_0}x_1^{j_1}x_2^{j_2}...x_{n-1}^{ j_{n-1} }\cdot a) \otimes \psi^{-1}((g^{j_0}x_1^{j_1}x_2^{j_2}...x_{n-1}^{ j_{n-1} })^{\ast})\\
		&=& 1\cdot a \otimes \psi^{-1}(1^{\ast})+\sum_{i=1}^{n-1}x_i\cdot a \otimes \psi^{-1}(x_i^{\ast})+gx_i\cdot a \otimes \psi^{-1}(gx_i^{\ast})\\
		&=&a\otimes \left(\frac{1+g}{2}\right)-\sum\limits_{i=1}^{n-1}w_ia\otimes gx_i, 
	\end{eqnarray*}
	for all $a\in A$, where $\psi$ is given by Lemma \ref{iso_nichols}. 
\end{proof}

\noindent{\textit{Remark.}} By the correspondence  $A_1\Leftrightarrow g \cdot 1_A$, $A_1=1_A$ if only if $\rho$ is global coaction.

\begin{exa}
	When $A=\Bbbk$, these partial coactions coincide with the partial coactions $\rho(a)=1\otimes z_\alpha$, where $z_\alpha= \frac{1+g}{2} - \sum_{i=1}^{n-1}\alpha_i \, gx_i$, presented in \cite[Theorem 4.6]{taft_corpo_revista}.
\end{exa}

\begin{exas}		
	If $n=2$ every partial coaction such that $A_1 =0$ and $w \in Z(A)$ is given by $\rho(a)=a\otimes \left(\frac{1+g}{2}\right)- wa\otimes gx,$
	for all $a\in A$. See Example \ref{exa_coactiontaft}.
	
	If $n=3$, every partial coaction such that $A_1 =0$ and $w_1, w_2\in Z(A)$ is given by $\rho(a)=a\otimes \left(\frac{1+g+g^2}{3}\right)- w_1 a\otimes gx_1-w_2a\otimes gx_2,$ for all $a\in A$.
\end{exas}

\bibliographystyle{abbrv}

\bibliography{referencias}

\end{document}